\documentclass{article}

\usepackage{amsfonts, amsmath, amssymb, amsthm, constants, bbm}
\usepackage{comment}
\newtheorem{thm}{Theorem}

\newtheorem{lem}[thm]{Lemma}
\newtheorem{cor}[thm]{Corollary}
\newtheorem{cond}[thm]{Condition}

\newcommand{\bl}{s}

\usepackage{xcolor}

\usepackage{xr}
\usepackage[]{graphicx}

\makeatletter
\newcommand*{\addFileDependency}[1]{
	\typeout{(#1)}
	\@addtofilelist{#1}
	\IfFileExists{#1}{}{\typeout{No file #1.}}
}
\makeatother

\newcommand*{\myexternaldocument}[1]{%
	\externaldocument{#1}%
	\addFileDependency{#1.tex}%
	\addFileDependency{#1.aux}%
}

\myexternaldocument{supp}

\usepackage{soul}

\begin{document}

	\title{Large-Sample Properties of Non-Stationary Source Separation for Gaussian Signals}

	\author{
			Fran\c{c}ois~Bachoc \\
		 Toulouse Mathematics Institute,		  University Paul Sabatier, France \\
		  ~ \\
		Christoph~Muehlmann \\
		Institute of Statistics \& Mathematical Methods in Economics, \\
		 Vienna University of Technology, Austria \\
		 ~ \\
		Klaus~Nordhausen \\
		Department of Mathematics and Statistics, \\
		 University of Jyv\"askyl\"a, Finland \\
		 ~ \\
		Joni~Virta  \\
		Department of Mathematics and Statistics, 		University of Turku, Finland
}

	\maketitle
	
	\begin{abstract}
		Non-stationary source separation is a well-established branch of blind source separation with many different methods. However, for none of these methods large-sample results are available. To bridge this gap, we develop large-sample theory for NSS-JD, a popular method of non-stationary source separation based on the joint diagonalization of block-wise covariance matrices. We work under an instantaneous linear mixing model for independent Gaussian non-stationary source signals together with a very general set of assumptions: besides boundedness conditions, the only assumptions we make are that the sources exhibit finite dependency and that their variance functions differ sufficiently to be asymptotically separable. The consistency of the unmixing estimator and its convergence to a limiting Gaussian distribution at the standard square root rate are shown to hold under the previous conditions. Simulation experiments are used to verify the theoretical results and to study the impact of block length on the separation.

		
		{\it This work has been submitted to the IEEE for possible publication. 
			Copyright may be transferred without notice, after which this 
			version may no longer be accessible.}
	\end{abstract}

	\section{Introduction}\label{sec:intro}
	
	The linear blind source separation (BSS) model assumes that a set of $p$-dimensional signals $ X_t $ is an instantaneous linear mixture of a set of $p$ source signals $ Z_t $,
	\begin{equation}\label{eq:bss:model}
		X_t = A Z_t, \quad t \in \mathbb{N},
	\end{equation}
	where the parameter of interest is the \textit{mixing matrix} $A \in \mathbb{R}^{p \times p}$ (or, equivalently, the \textit{unmixing matrix} $ A^{-1} $) which is assumed to be invertible, see \cite{ComonJutten2010}. In practice one observes the series $X_t$ at the instances $1, \ldots , T$, for some $T \in \mathbb{N}$.

	Generally, the latent signals in $Z_t$ are assumed to exhibit a dependency structure that is simpler than the one existing between the observed signals in $X_t$ (often temporal uncorrelatedness or full independence). This assumption makes the model \eqref{eq:bss:model} especially attractive in modelling and prediction where the two tasks are made considerably easier after the estimation of the matrix $A$. This is because, by working with $ Z_t $, we avoid the need for modelling the dependencies between the variables, which would be necessary if we operated directly on the observed $ X_t $ \cite{ComonJutten2010,CichockiAmari2002,PhamCardoso2001}. Among the most common assumptions on the dependency structure of $ Z_t $ is second order stationarity coupled with the fact that the autocovariance structures of the signals are sufficiently different to be distinguishable from each other \cite{ComonJutten2010,PanMatilainenTaskinenNordhausen2022}. The resulting methodology is called second order source separation (SOS), incorporating classical methods such as AMUSE \cite{TongSoonHuangLiu1990} and SOBI \cite{BelouchraniAbedMeraimCardosoMoulines1997}, which are based on the diagonalization of autocovariance matrices, and also more recent ones, see, e.g., \cite{MiettinenNordhausenOjaTaskinen2014b,MiettinenMatilainenNordhausenTaskinen2020}.
	
	The stationarity (and ergodicity) of the series are convenient assumptions also with respect to large-sample statistics. Under them, the sample moments of the series can be expected to converge to their population values \cite{BrockwellDavis1991}, often a key requirement in studying the limiting behavior of unmixing estimators. In this work, we step outside of this standard asymptotic framework and develop large-sample theory under the model~\eqref{eq:bss:model} and the assumption of non-stationary sources. The resulting \textit{non-stationary source separation} (NSS) model is highly appealing in many applications, such as speech recognition, where the signals cannot be expected to be stationary, but rather piece-wise stationary \cite{nassif2019speech}  {or in a group independent component analysis (ICA) framework in which the data from $N$ subjects are concatenated. In group ICA it is assumed that model~\eqref{eq:bss:model} holds for each subject with the same mixing matrix $A$ but that the sources might have slightly different properties~\cite{PfisterWeichwaldBuehlmannSchoelkopf2019}}.
	
	Three standard methods of estimating $ A^{-1} $ in NSS are known as NSS-SD, NSS-JD and NSS-JD-TD, \cite{ChoiCichocki_2000_IEEE,choi2000blind,ChoiCichockiBelouchrani2001}. Each of the methods is based on dividing the total observed $ T $-length time series into $ K $ blocks and jointly diagonalizing a set of block-wise covariance or autocovariance matrices. More precisely, NSS-SD (simultaneous diagonalization) uses $ K = 2 $ blocks and simultaneously diagonalizes the block-wise covariance matrices of the two blocks. NSS-JD (joint diagonalization) instead first whitens the series using the global covariance matrix, and then jointly diagonalizes the block-wise covariance matrices of an arbitrary amount $ K $ of blocks. NSS-TD-JD (time-delayed joint diagonalization) is otherwise as NSS-JD, but includes also block-wise autocovariance matrices (for a suitable set of lags) in the joint diagonalization. Each method uses successively more information on the source series than the previous.  {NSS-SD and NSS-JD basically just need ordered observations but do not need serial dependence while NSS-TD-JD utilizes also information on the time dependency structure of the series and is usually considered in a block stationary framework. 
		
		In this work our focus is on the theoretical properties of the second method, NSS-JD, which is presented in more detail in Section \ref{sec:nss}. Our reason for focusing on NSS-JD is two-fold. On one hand, it is in most simulation studies, such as in \cite{ChoiCichocki_2000_IEEE,Nordhausen2014}, superior to NSS-SD and, on the other hand, it is more general than NSS-TD-JD which requires block stationary structures.
		
		Recent developments in non-stationary source separation include \cite{PhamCardoso2001,HsiehChien2011,Nordhausen2014, PfisterWeichwaldBuehlmannSchoelkopf2019,TehraniSameiJutten2020} which consider, for example, robust and Bayesian approaches or assume that not all components are non-stationary.}  {However, asymptotic considerations for NSS are so far still missing. Possibly the most methodological approach is given in \cite{PhamCardoso2001} which develops NSS approaches embedded in a Gaussian maximum likelihood framework and a Gaussian mutual information framework assuming independent observations. For these approaches either a block stationary model is assumed or some smooth function for the changing component variance needs to be modeled. It is also noteworthy that the choice of the number of blocks $K$ is hardly ever discussed and in most papers mentioned above $K$ is usually chosen in such a way that it contains 50 or 100 observations. Simulations in \cite{Nordhausen2014} indicate however that a sufficient number of blocks seems more relevant than the number of observations within a block.
		
		Taking the above considerations into account, the large-sample properties we develop are non-standard in the sense that, due to stating very weak structural assumptions on the source series, the limiting distributions of our estimators are not static, but instead change with the sample size $ T $.  
		Another key property of our framework is that we will not take the number of blocks $K$ as fixed but instead let $ K \rightarrow \infty $, proportional to $ T $.
		Our method of proof uses a parametrization of the space of orthogonal matrices through matrix exponentials of skew-symmetric matrices. This enables us to simplify the analysis of $ M $-estimators that are orthogonal matrices. This technique may prove useful also in other BSS problems. The previous points mean that extra layers of complexity arise in the theory, and as such, the derivation of our theoretical results is postponed to the online supplementary material.
		
		The structure of the paper is as follows. In Section \ref{sec:nss} we go over the non-stationary source separation model and recall one of the standard estimators for its unmixing matrix, NSS-JD. In Section \ref{sec:large:sample} we give our main results regarding the consistency and limiting normality of the estimator, along with the assumptions required for the results to hold. We also discuss the strictness of the assumptions. In Section \ref{sec:simulation} we apply the method to simulated data and demonstrate that the asymptotic results are representative of the finite sample behavior. In Section \ref{sec:conclusion} we close with a discussion of possible extensions for future study and, finally, the covariance matrix of the limiting distribution derived in Section \ref{sec:large:sample} is presented in Appendix \ref{sec:appendix}.
		
		\section{Non-stationary source separation}\label{sec:nss}
		
		In this section we review the non-stationary source separation model along with NSS-JD, a method for non-stationary source estimation. As described in Section \ref{sec:intro}, NSS-JD is based on dividing the total observed time series into blocks, and throughout the paper we assume, for convenience, that the blocks are of equal, fixed length, denoted in the following by $ \bl \geq 2 $. Hence, $T = K \bl$ throughout, and $K \to \infty$ with $\bl$ fixed in all the large-sample considerations. An extension to varying (but bounded) block lengths would be straightforward, but tedious in notation.
		
		\subsection{NSS model}\label{sec::nss_model}
		
		Recall from Section \ref{sec:intro} that we observe the instantaneous linear mixing model,
		\begin{align}\label{eq:nss:model}
			X_t = A Z_t, \quad t \in \{ 1, \ldots , K \bl \},
		\end{align}
		where $K \in \mathbb{N}$ is a positive integer and $ A \in \mathbb{R}^{p \times p} $ is invertible. In Section \ref{sec:large:sample} we will detail the exact assumptions for \eqref{eq:nss:model} that are required for our large-sample results to hold, in particular, the independence and Gaussianity of the source series. However, recall that no stationarity assumptions are made for $ Z_t $, meaning that both the mean and variance functions of $ Z_t $ are allowed to be non-constant and arbitrary (but bounded). Additionally, we will also postpone discussing the identifiability of the model parameters to Section \ref{sec:large:sample}. 
		
		Finally, note that we impose in \eqref{eq:nss:model}, for convenience, the assumption that the total observed time series length $ T = K \bl $ is a multiple of the fixed block length $ \bl $, meaning that we have exactly $ K $ blocks. This assumption is completely without loss of generality, as including a finite ``tail'' of observations, $ T = K \bl + r $, $ r \in \{ 0, \ldots \bl - 1 \} $, has no impact in the asymptotic regime $ K \rightarrow \infty $ we pursue in Section \ref{sec:large:sample}.
		
		\subsection{NSS-JD estimate of the unmixing matrix}\label{sec:nss_jd}
		
		To estimate the unmixing matrix $ A^{-1} $ we use NSS-JD which is based on the simultaneous diagonalization of block-wise covariance matrices. Let
		\begin{align}\label{eq:cov:x:1}
			\hat{\mathrm{cov}}_{X,i}
			= \frac{1}{\bl} \sum_{j=1}^\bl (X_{\bl(i-1)+j} - \bar{X}_i) (X_{\bl(i-1)+j} - \bar{X}_i)'
		\end{align}
		denote the covariance matrix of the $ i $th block of length $ \bl $ where $ \bar{X}_i = (1/\bl) \sum_{j = 1}^\bl X_{\bl(i-1)+j} $ is the sample mean vector of the $ i $th block, $ i \in \{ 1, \ldots , K \} $. The subscript $ X $ in $\hat{\mathrm{cov}}_{X,i}$ is used to differentiate from the analogous quantities defined in Section \ref{sec:large:sample} for the latent series $ Z_t $. Let further 
		\begin{align}\label{eq:cov:x:2}
			\hat{\bar{\mathrm{cov}}}_{X,K} = \frac{1}{K} \sum_{i=1}^K \hat{\mathrm{cov}}_{X,i}
		\end{align}
		denote the average block-wise covariance matrix over all $ K $ blocks. Note that in $ \hat{\bar{\mathrm{cov}}}_{X,K} $ the centering is done block-wise, and thus $ \hat{\bar{\mathrm{cov}}}_{X,K} $ is not equal to the usual sample covariance matrix of the full series where the centering is with respect to the global mean vector $ \bar{X} = (1/T) \sum_{i = 1}^T X_{i} $. This modification is necessary as our theoretical results are based on exploiting finite dependence within individual series, which global centering would break. 
		
		The NSS-JD estimate of the unmixing matrix is now found as
		\begin{align}\label{eq:nssjd_solution}
			\hat{W}_{X,K} = \hat{U}_{X,K} \hat{\bar{\mathrm{cov}}}_{X,K}^{-1/2}.
		\end{align}
		Here $ \hat{\bar{\mathrm{cov}}}_{X,K}^{-1/2} $ denotes the unique symmetric positive definite matrix satisfying $ \hat{\bar{\mathrm{cov}}}_{X,K}^{-1/2} \hat{\bar{\mathrm{cov}}}_{X,K} \hat{\bar{\mathrm{cov}}}_{X,K}^{-1/2} = I_p $, and $ \hat{U}_{X,K} $ is the \textit{joint diagonalizer} of the block-wise covariance matrices of the series whitened by the average block-wise covariance matrix $ \hat{\bar{\mathrm{cov}}}_{X,K} $. By joint diagonalizer, we refer to a solution of the following optimization problem,
		\begin{align}\label{eq:joint_diag_1}
			\hat{U}_{X,K} 
			\in
			\mathrm{argmax}_{ U \in \mathcal{O}_p } g(U)
			,
		\end{align}
		where $ \mathcal{O}_p $ is the set of $ p \times p $ orthogonal matrices, and
		\begin{align}\label{eq:joint_diag_2}
			g(U) = \sum_{i=1}^K
			||  
			\mathrm{diag}
			\left(
			U
			\hat{\bar{\mathrm{cov}}}_{X,K}^{-1/2}
			\hat{\mathrm{cov}}_{X,i}
			\hat{\bar{\mathrm{cov}}}_{X,K}^{-1/2}
			U'
			\right)
			||^2,
		\end{align}
		where $ \mathrm{diag}(A) $ denotes the diagonal matrix having the same diagonal elements as $ A $ and $||A||$ is the Frobenius norm of $A$, for any square matrix $A$. The set notation in \eqref{eq:joint_diag_1} is justified as the maximizer of $ g $ can never be unique since any optimal $ U $ can always have its rows permuted or their signs changed to produce a distinct optimal solution. Calling \eqref{eq:joint_diag_2} joint diagonalization is confirmed by the orthogonal invariance of the Frobenius norm. I.e.: maximizing $ g $ is equivalent to minimizing the sum of the squared off-diagonal elements of $ \hat{\bar{\mathrm{cov}}}_{X,K}^{-1/2}
		\hat{\mathrm{cov}}_{X,i}
		\hat{\bar{\mathrm{cov}}}_{X,K}^{-1/2} $ for $ i \in \{ 1, \ldots , K \} $. See \cite{IllnerMiettinenFuchsTaskinenNordhausenOjaTheis2015} for different algorithms for solving \eqref{eq:joint_diag_1}. In the simulations of Section \ref{sec:simulation} we use the standard algorithm based on Jacobi rotations \cite{Clarkson1988, BelouchraniAbedMeraimCardosoMoulines1997}.
		
		Given the unmixing estimate $ \hat{W}_{X,K} $, an estimate of the latent sources $Z_t$ is given by $ \hat{W}_{X,K} X_t $. In Section \ref{sec:large:sample} we show that this indeed provides a consistent estimate of the sources, under a general set of assumptions. Note finally, that the estimate produced by NSS-JD has the following invariance property \cite{Nordhausen2014,MiettinenTaskinenNordhausenOja2015}: Let $ L \in \mathbb{R}^{p \times p} $ be an arbitrary invertible matrix and let $ \hat{W}_{LX,K}, \hat{U}_{LX,K} $ etc. be defined as above but with $ X_t $ replaced by $ LX_t $. Then, the set of source estimates $ \{ \hat{W}_{X,K} X_t \} $, where $ \hat{U}_{X,K} $ goes through all maximizers \eqref{eq:joint_diag_1}, is equal to the set of source estimates $ \{ \hat{W}_{LX,K} L X_t \} $, where $ \hat{U}_{LX,K} $ goes through all maximizers of the equivalent of \eqref{eq:joint_diag_1} for $ LX_t $. As such, changing the coordinate system of the observations has no effect on the produced (set of) source estimates and we may, without loss of generality, restrict to $ A = I_p $ later on in the simulations in Section \ref{sec:simulation}.

		\section{Large-sample properties of NSS-JD}\label{sec:large:sample}
		
		We divide the discussion of the large-sample properties of NSS-JD into two parts, first going over the required assumptions and then stating the results on consistency and limiting normality.
		
		\subsection{Assumptions}\label{sec:assumptions}
		
		Assume that the observations obey the NSS model \eqref{eq:nss:model} and that we work in the asymptotic regime that $ K \rightarrow \infty $. Thus the block length $ \bl $ is kept fixed but the sample size increases by including more and more blocks in the joint diagonalization.
		
		The following list details the assumptions required for the consistency and the limiting normality of the NSS-JD estimator to hold. Along the assumptions we also discuss their intuitive meanings.
		
		\begin{cond} \label{cond:gaussian:independent}
			Denoting $Z_t = (Z_{t}^{(1)}, \ldots , Z_{t}^{(p)})'$, the latent sources $ Z_{t}^{(1)}, \ldots , Z_{t}^{(p)} $ are independent jointly Gaussian processes.
		\end{cond}
		As is common in the literature, we make, for technical convenience, the assumption of Gaussian sources. This in turn means that the latent sources, which are uncorrelated by definition, are also independent.
		
		\begin{cond} \label{cond:block:independence}
			There exists a fixed $L \in \mathbb{N}$ such that
			for any $i,j \in \{1, \ldots ,K\}$, $ |i - j| \geq L$, and $ k \in \{ 1, \ldots , p \} $, the vectors $( Z_{\bl(i-1)+a}^{(k)} )_{a=1,\ldots,\bl}$ and $( Z_{\bl(j-1)+a}^{(k)} )_{a=1,\ldots,s}$ are independent.
		\end{cond}
		Condition \ref{cond:block:independence} states that the latent series exhibit finite dependency. Note that the length of the dependency can be arbitrary (as long as its finite) without affecting the conclusions of subsequent Theorems \ref{theo:consistency} and \ref{theo:limiting_normality} but we expect that, the longer the memory, the larger the asymptotic variances of the estimators are.
		
		\begin{cond} \label{cond:mean:function}
			For any $i \in \{1,\ldots, K\}$ and $j \in \{1,\ldots,s\}$,  the mean vector of $Z_{\bl(i-1)+j}$ depends only on $i$.
		\end{cond}
		The simplest way to fulfill Condition \ref{cond:mean:function} is to simply assume a stationary mean for the latent vectors $ Z_t $, which is standard in the NSS literature \cite{CichockiAmari2002,PhamCardoso2001,PanMatilainenTaskinenNordhausen2022}. However, this is not necessary and if, e.g., the series are \textit{a priori} known to have block-wise constant means with some block length $ s_0 $, the block length $ s $ can be chosen to be equal to this, fulfilling Condition \ref{cond:mean:function}.  {This is for example natural to assume in a group ICA framework~\cite{PfisterWeichwaldBuehlmannSchoelkopf2019}.}
		
		\begin{cond} \label{cond:max:variance:bis}
			We have
			\[
			\sup_{t \in \mathbb{N}}
			\max_{ j=1, \ldots ,p }
			\mathrm{Var}
			\left(
			Z_{t}^{(j)}
			\right)
			\leq
			C,
			\]
			for some fixed $ C \in \mathbb{R} $.
		\end{cond}
		Condition \ref{cond:max:variance:bis} is a technical assumption that simply requires that the variance functions of the latent sources are bounded.
		This assumption is of course much weaker than assuming the sources to be stationary.
		Note that, without boundedness, even whitening would become asymptotically infeasible.
		
		Denote the population block-wise covariance matrices in the following by
		\begin{align}\label{eq:cov:z:1} 
			\mathrm{cov}_{Z,i}
			= \frac{1}{\bl} \sum_{j=1}^\bl \mathbb{E}\left[(Z_{\bl(i-1)+j} - \bar{Z}_i) (Z_{\bl(i-1)+j} - \bar{Z}_i)'\right]
		\end{align}
		with $\bar{Z}_i = (1/s) \sum_{j=1}^s Z_{\bl(i-1)+j}$,
		for $ i \in \{ 1, \ldots , K \} $, and their average by,
		\[ 
		\bar{\mathrm{cov}}_{Z,K}
		= \frac{1}{K } \sum_{i=1}^K \mathrm{cov}_{Z,i}.
		\]
		
		\begin{cond} \label{cond:asymptotically:distinct:eigenvalues}
			There exists a strictly increasing sequence $(i_k)_{k \in \mathbb{N}}$, such that $i_k \in \mathbb{N}$ for all $k \in \mathbb{N}$ and such that, with $N_K = \# \{ k \in \mathbb{N} ; i_k \leq K \}$,
			we have $\liminf  N_K /K >0$ as $K \to \infty$. There exists a fixed $\delta >0$, such that
			\[
			\inf_{ \substack{k \in \mathbb{N}  }}
			\min_{ \substack{ i,j=1,\ldots,p \\ i \neq j } }
			\left|
			\left [\mathrm{cov}_{Z,i_k} \right]_{i,i}
			-
			\left [\mathrm{cov}_{Z,i_k} \right]_{j,j}
			\right|
			\geq \delta
			\]
			and 
			\[
			\inf_{ \substack{k \in \mathbb{N}  }}
			\min_{ \substack{ i=1,\ldots,p  } }
			\left [\mathrm{cov}_{Z,i_k}\right]_{i,i}
			\geq \delta.
			\]
		\end{cond}
		Condition \ref{cond:asymptotically:distinct:eigenvalues} is what guarantees that the latent sources are asymptotically separable from each other, by requiring that within a positive fraction of blocks the sources have different enough variance structures to be distinguishable from each other. This condition can be interpreted as an extension of the requirement that the eigenvalues of a matrix are distinct (in order to produce a unique set of eigenvectors), to the case of growing number of matrices.
		
		Let us finally introduce a convention on $ Z_t $ to make the unmixing matrix more identifiable. 
		By Conditions \ref{cond:gaussian:independent} and \ref{cond:mean:function}, the matrix $ \bar{\mathrm{cov}}_{Z,K} $ is diagonal and has the average variances of the (empirically centered) sources as its diagonal elements. Now, the scales of the sources in the model \eqref{eq:nss:model} are confounded with the magnitudes of the columns of $ A $ (one may multiply any of the sources by $ \lambda \neq 0 $ and the corresponding columns of $ A $ by $ 1/\lambda $ without changing the model). Thus, without loss of generality, we fix $ \bar{\mathrm{cov}}_{Z,K} = I_p $ throughout the rest of the paper. Under Condition \ref{cond:asymptotically:distinct:eigenvalues}, this makes the unmixing matrix, for large enough $ n $, identifiable up to row permutation and multiplication by $ \pm 1 $.
		
		\subsection{Consistency and limiting normality}
		
		After Conditions \ref{cond:gaussian:independent}--\ref{cond:asymptotically:distinct:eigenvalues} and fixing the scales of the sources through $ \bar{\mathrm{cov}}_{Z, K} = I_p $, the ordering and the signs of the sources can still be chosen freely. This is usually thought to be acceptable in practice, as after the extraction of the sources, subsequent univariate analyses can be used to assess their relative importance. To accommodate this indeterminacy in the following results, let $ \mathcal{G}_p $ denote the set of all signed permutation matrices ($ p \times p $ matrices with a single $ \pm 1 $ in each row and column and rest of the elements zero).
		
		\begin{thm}\label{theo:consistency}
			Assume that Conditions \ref{cond:gaussian:independent}--\ref{cond:asymptotically:distinct:eigenvalues} hold.
			Then, for any sequence $\hat{W}_{X,K}$ of NSS-JD estimates defined in \eqref{eq:nssjd_solution} and \eqref{eq:joint_diag_1}, there exists a sequence $\hat{G}_K \in \mathcal{G}_p$ such that
			\[
			\hat{G}_K \hat{W}_{X,K}
			\underset{K \to \infty}{\overset{p}{\to}}
			A^{-1}.
			\] 
		\end{thm}

		\begin{thm}\label{theo:limiting_normality}
			Assume that Conditions \ref{cond:gaussian:independent}--\ref{cond:asymptotically:distinct:eigenvalues} hold. Then, for any sequence $\hat{W}_{X,K}$ of NSS-JD estimates defined in \eqref{eq:nssjd_solution} and \eqref{eq:joint_diag_1}, there exists a sequence $\hat{G}_K \in \mathcal{G}_p$ such that, with $Q_{\hat{W}_{X,K} }$ the distribution of $ K^{1/2} ( \hat{G}_K \hat{W}_{X,K} - A^{-1} )$, we have
			\[
			d_w 
			\left( 
			Q_{\hat{W}_{X,K} } , \ \mathcal{N}(0, \Sigma_{ \hat{W}_{X,K} }) 
			\right)
			\underset{K \to \infty}{\overset{}{\to}}
			0,
			\]
			where $ d_w $ denotes a metric generating the topology of weak convergence on the set of Borel
			probability measures on Euclidean spaces, and the limiting covariance matrix $ \Sigma_{ \hat{W}_{X,K} } $ is bounded as $K \to \infty$. The exact form of $ \Sigma_{ \hat{W}_{X,K} } $ is given in Appendix~\ref{sec:appendix}.
		\end{thm}
		
		In Theorems \ref{theo:consistency} and \ref{theo:limiting_normality}, the statement ``for any sequence $\hat{W}_{X,K}$'' refers to the fact that the maximizer of \eqref{eq:joint_diag_1} is not unique. As described in Section \ref{sec:intro}, the lack of stationarity and structural assumptions implies that the approximating distribution of $ \hat{W}_{X,K} $ is not static but instead evolves with $ K $. This requires us to express the result of Theorem \ref{theo:limiting_normality} using the metric $ d_w $, instead of the more standard convergence in distribution. See, e.g., the discussion in \cite[p. 393]{dudleyreal} for specific examples on the use of the metric $ d_w $. See also \cite{bachoc2019spatial}.
		
		We also remark that limiting normality of a random matrix means that its (row or column) vectorization  converges in distribution to a Gaussian vector. Furthermore, the asymptotic covariance matrix $ \Sigma_{ \hat{W}_{X,K} } $ has dimension $p^2 \times p^2$ and is convenient to express with quadruple indices, see Appendix~\ref{sec:appendix}.
		
		To conclude this section we point out that, as with the original NSS-JD, our variant of NSS-JD is a valid NSS method also for non-Gaussian data and processes. The assumption of Gaussianity in this work is solely made for the sake of deriving the former large-sample behavior of the estimator.


		\section{Simulation studies}\label{sec:simulation}
		
		\begin{figure*}[!t] 
			\centering 
			\includegraphics[width=0.95\linewidth]{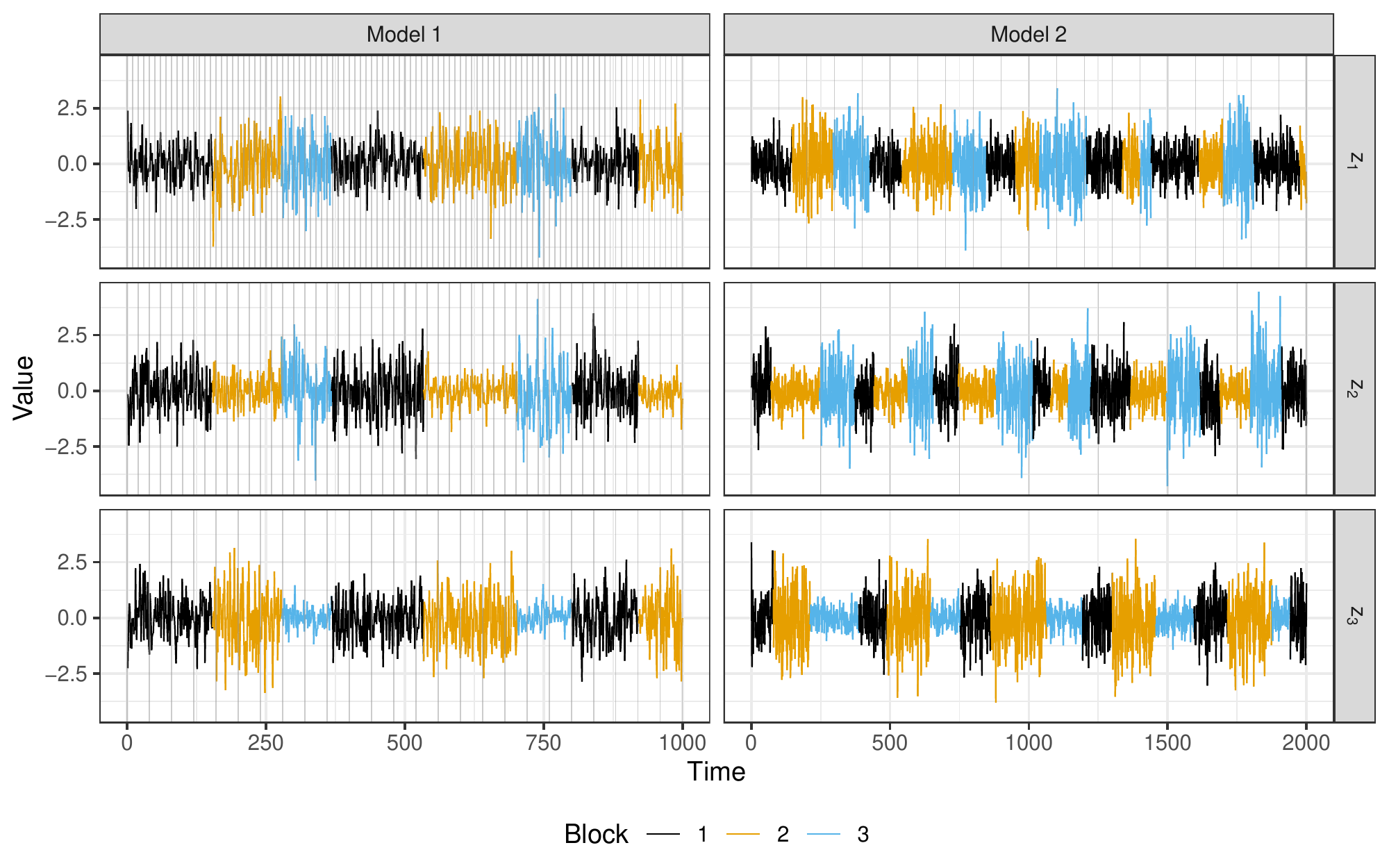} 
			\caption{Example latent time series for Model 1 of length 1000 and Model 2 of length 2000. The vertical gray lines depict the different choices for the block size of the estimator. Namely, $\bl$ = 10, 20 and 40 from top to bottom in the panels for Model 1 and $\bl$ = 100 and 250 for the first two panels of Model 2.}
			\label{fig:model_1_2} 
		\end{figure*}
		
		\begin{figure*}[!t] 
			\centering 
			\includegraphics[width=0.95\linewidth]{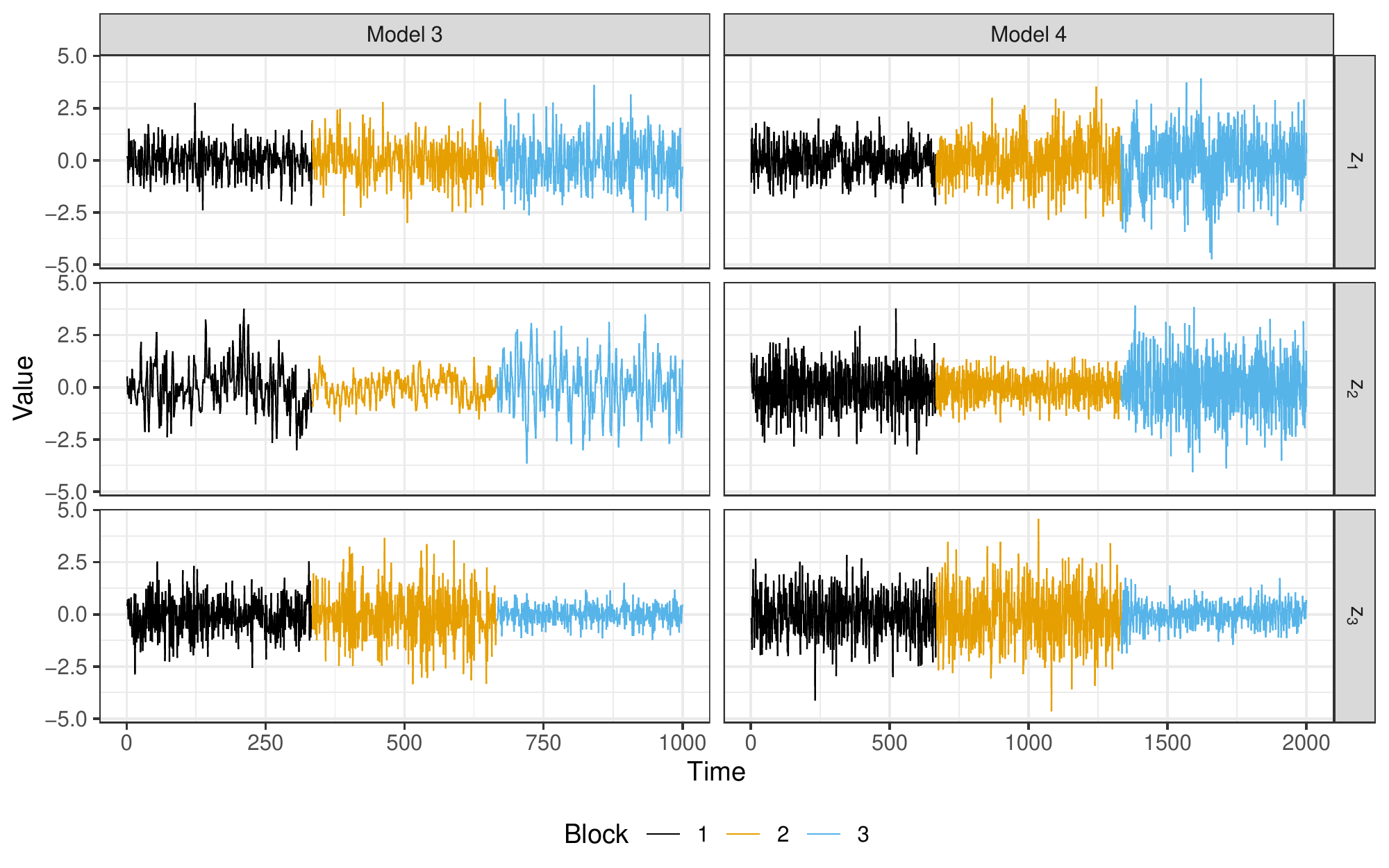} 
			\caption{Example latent time series for Model 3 of length 1000 and Model 4 of length 2000.}
			\label{fig:model_3_4} 
		\end{figure*}
		
		\begin{figure*}[!t] 
			\centering 
			\includegraphics[width=0.95\linewidth]{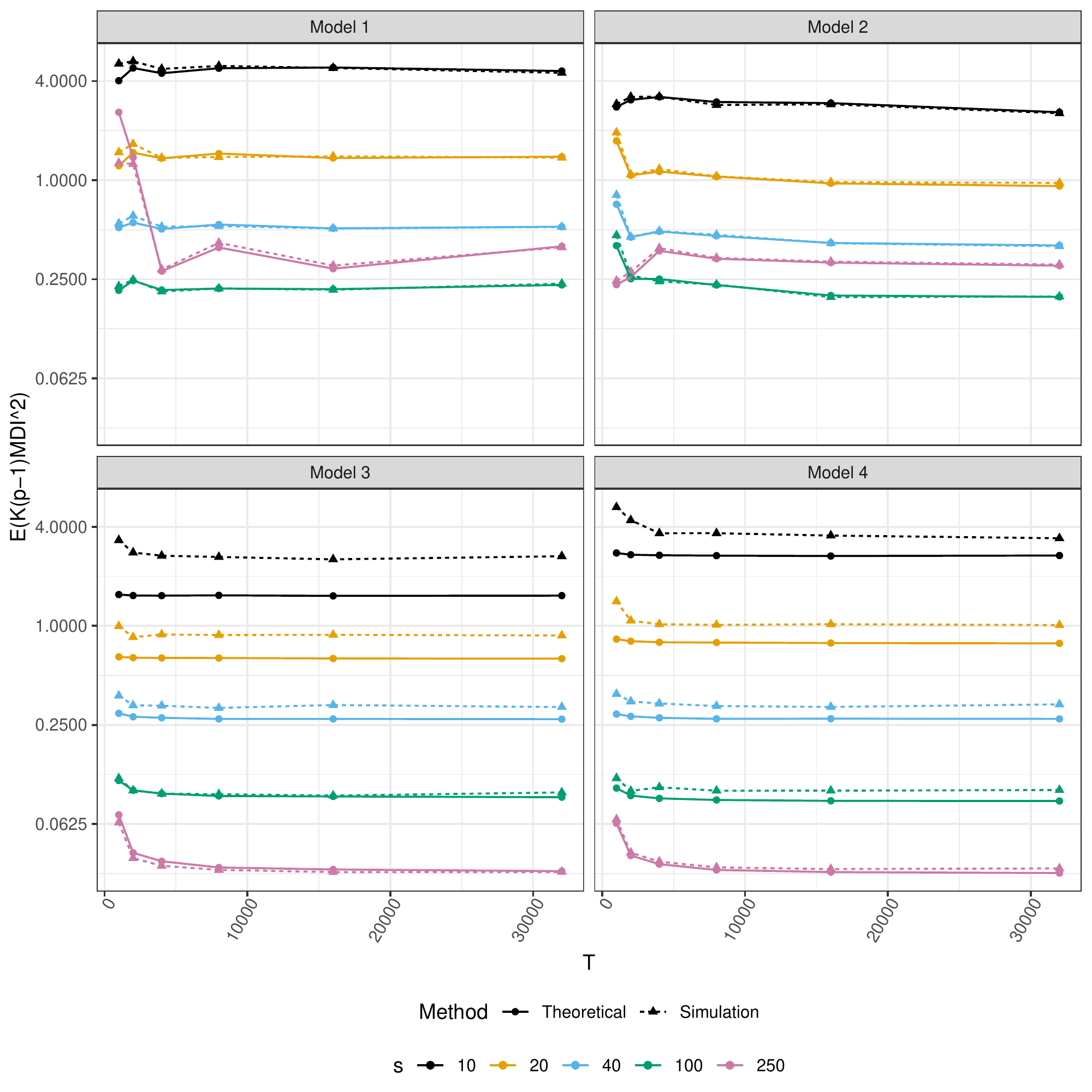} 
			\caption{Results for the expectation of the limiting distribution of the quantity $K(p-1) \text{MDI}(\hat{W}) ^ 2$ (theoretical curves) and simulations of the former quantity for all considered sample sizes, block sizes and models.}
			\label{fig:MDI} 
		\end{figure*}
		
		In this part of the paper we carry out an extensive simulation study to verify the derived asymptotic results. The simulation study is carried out in R 4.0.3 \cite{R_software} with the help of the package JADE \cite{JADE_package}.
		
		We consider Gaussian time series following the NSS model  {\eqref{eq:nss:model}} where the dimension of the time series $p$ equals $3$ and the mixing matrix is chosen to be the unit matrix, the latter choice is justified due to the affine equivariance property of the NSS estimator (details are given in  {Section~\ref{sec:nss_jd}}). The lengths of the time series equal $T=2^n 1000$ where $n=0,1,2,3,4,5$ and the block sizes are  {$\bl=10,20,40,100,250$}. For the latent time series we choose a total of four models: Models~1 and 2 are  {models with independent observations} (no serial dependence but time varying variances) and Models~3 and 4 are based on moving average processes with time varying innovation variance.
		
		Specifically, for Model~1 the latent time series is formed by concatenating differently sized blocks until the overall length $T$ is reached, the block lengths are independent samples from the negative binomial distribution  {$\mathcal{NB} (6, 1/20)$}. In each block the variances will be different for each component of the latent time series repeating itself every third block. For the first component the variances equal 1, 2 and 3, for the second component 3, 1 and 5 and for the third component 4, 7 and 1. I.e., the first block of the first component of the latent time series has variance 1, the second block has variance 2, the third block has variance 3, the fourth block again variance 1 and so fourth. Model~2 is equal to Model~1 where the only difference is given by the fact that the block sizes are randomly sampled for each component of the latent time series individually, again from  {$\mathcal{NB} (6, 1/20)$}. Figure~\ref{fig:model_1_2} depicts samples from these models, where the different blocks are highlighted by different colors.
		
		For Model~3 each component of the latent time series follows a moving average process where the coefficients are  {$(0.9, -0.8, 0.3, -0.5)$} for the first component,  {$(0.8, 0.2, 0.3)$} for the second component and  {$(-0.6, 0.7, 0.1)$} for the third one. For the innovations of the  {$\mathrm{MA}$} processes the latent time series is split into three equally-sized blocks, the variances of the innovations are chosen to be  {the} ones of Model~1 for each of the three blocks. Thus, Model~3 can be viewed as a  block stationary model in the sense that in each block the time series is weakly stationary. Model~4 follows the same principle as Model~3, only the  {$\mathrm{MA}$} processes for the latent time series are chosen to be  {$\mathrm{MA}(40)$, $\mathrm{MA}(50)$ and $\mathrm{MA}(60)$} where the coefficients for the processes are determined by one iid sample from the uniform distribution $U(-1,1)$. Figure~\ref{fig:model_3_4} illustrates samples from Model~3 and Model~4. Note that for all of the four models the time series are scaled such that they fulfill the unit covariance condition  {(i.e., $ \bar{\mathrm{cov}}_{Z,K} = I_p$, see the end of Section~\ref{sec:assumptions})}. 
		
		For an unmixing matrix $\hat{W}$ estimated by the NSS method, an indicator for the quality of the estimation can be based on the fact that $\hat{W} A \approx I_p$ up to the model indeterminacies of sign and permutation (the scale is already fixed). One quantity that is based on the former consideration is the minimum distance index (MDI) \cite{IlmonenNordhausenOjaOllila2010,LietzenVirtaNordhausenIlmonen2020}. The MDI is defined by
		\[
		\text{MDI}(\hat{W}) = \frac{1}{\sqrt{p-1}} \inf_{G \in \mathcal{G}_p} \| G \hat{W}  A -  I_p \|_F .    
		\]
		
		For a perfect separation it holds that $\hat{W} A = I_p$ (up to sign and permutation) which leads to an MDI of zero (lower limit), as the signal separation gets worse the MDI approaches its upper limit of one. Furthermore, for an estimator that follows a central limit theorem in the sense  {of Theorem~\ref{theo:limiting_normality},} the adapted MDI $K(p-1) \text{MDI}(\hat{W}) ^ 2$ converges in distribution to $\sum_{i=1}^m \delta_i \chi^2_i$, where $\chi^2_i, i=1,\ldots,m,$ are independent chi-squared random variables with one degree of freedom, and $\delta_i, i=1,\ldots,m,$ are the $m$ non-zero eigenvalues of some matrix dependent on  {the asymptotic covariance matrix of the estimator} (details can be found for example in \cite{IlmonenNordhausenOjaOllila2010}). This result leads to the fact that the expectation of the limiting distribution  {of the MDI} is given by the sum of all off-diagonal elements of  {$\Sigma_{ \hat{W}_{X,K} }$ from Theorem~\ref{theo:limiting_normality}}. Therefore, the asymptotic efficiency of the NSS method can be conveniently characterized by two numbers, namely, the expectation of the limiting distribution of the adapted MDI $K(p-1) \text{MDI}(\hat{W}) ^ 2$ (based on the limiting covariance matrix $\Sigma_{ \hat{W}_{X,K} }$) versus the average value of the adapted MDI based on several simulation repetitions. Figure~\ref{fig:MDI} illustrates these two numbers for all combinations of sample sizes, block sizes and models. The involved expectations  {in the expression of $\Sigma_{ \hat{W}_{X,K} }$ in Appendix~\ref{sec:appendix} for} the theoretical curves are based on 100000 Monte Carlo simulations and the simulated curves are based on 2000 repetitions.  {From the simulation results (Figure~\ref{fig:MDI}) we conclude the following points:}
		{\begin{itemize}
				\item For Model~1 and Model~2 the experimental lines (dashed) and the theoretical ones (solid) agree perfectly. For block length $s = 10,20,40,100$ the performance is stable and increasing with larger block size. However, for the largest block size $s=250$ the performance is less stable and is also worse than for $s=100$.
				\item For Model~3 and Model~4 convergence of the finite sample performance to the expected asymptotic level is much slower and only achieved for $s\geq 100$. However for all block lengths $s$ considered, the performance is stable and improves with increasing block length. Note also that in general a better separation seems possible in Model~3 and Model~4 compared to the other models.
			\end{itemize}
			Therefore we can conclude from our simulation study that the block length has a significant effect. It seems important that within a block the effective sample size is sufficient to estimate the covariance matrices with enough precision and therefore in cases with no or little dependence short block lengths are acceptable while with increasing dependence in the data the block lengths should be larger. However, the blocks should also not be made too large as then, it seems that there are not enough blocks to capture the non-stationarity features and thus the performance starts to suffer. This behavior is seen in Model~1 and Model~2. Moreover, it is worth noting that when the information within a block is sufficient, the convergence to the asymptotic limit is reached already for quite low sample sizes. 
	}}
	
	\section{Conclusion}\label{sec:conclusion}
	
	In this paper, we studied the large-sample properties of NSS-JD, a method of non-stationary source separation, under the unconventional asymptotic framework that the number of blocks grows without bounds, $ K \rightarrow \infty $, while the block size is kept fixed. Both consistency and limiting normality were shown to hold for the NSS-JD unmixing estimator under this framework.
	
	Although Conditions \ref{cond:gaussian:independent}--\ref{cond:asymptotically:distinct:eigenvalues} pose rather light restrictions on the source signals, extensions to at least two directions may prove feasible. First, while the assumption of Gaussian signals is a standard one in NSS, in applications such as finance more heavy-tailed distributions might prove a better choice. To accommodate this, the latent signals could be assumed to have block-wise elliptical distributions
	, the family of elliptical distributions preserving some key properties of the Gaussian family used in proving Theorems \ref{theo:consistency} and \ref{theo:limiting_normality}. Second, the finite dependency imposed by Condition \ref{cond:block:independence} could be replaced by assuming, e.g., exponentially decaying second-order dependence. In the spatial statistics literature, it is indeed common to consider (stationary) covariance functions that are not compactly supported but decrease exponentially fast to zero with the distance \cite{bachoc2014asymptotic,mardia1984maximum,bachoc2018asymptotic}. Some of the proof techniques used in these latter references could be beneficial to alleviate the finite dependency condition in our setting.

	Besides simply extending the method, an interesting follow-up to the current work would be to combine NSS-JD with latent dimension estimation. Namely, the BSS-model \eqref{eq:bss:model} is often combined with the assumption that the majority of the sources are noise, and the objective is to estimate only the non-noise sources, leading into a form of dimension reduction. In the NSS context, an appropriate definition of ``noise'' would be to define all second-order stationary sources to be noise, since they do not exhibit any changes in volatility over time. To separate the noise sources from the signal sources, note that, for the block-wise covariance matrices of the sources,
	the diagonal elements corresponding to the noise series are constant in expectation over the blocks.  {Thus, the sample variances of the eigenvalues over all blocks could be used to construct an asymptotic hypothesis test for the null hypothesis that some particular index of sources is noise. Similar strategies have been used for latent dimension estimation in unsupervised dimension reduction of iid data \cite{NordhausenOjaTylerVirta2017,LuoLi2016,NordhausenOjaTyler2022}, and second-order source separation \cite{MatilainenNordhausenVirta2017,VirtaNordhausen2021}. Some first steps in this direction in an NSS context are made in \cite{TehraniSameiJutten2020,SameniJutten2021}. This is also closely connected to stationary subspace analysis (SSA) where the goal is to separate the stationary subspace of a multivariate time series from its non-stationary subspace \cite{BunauMeineckeKiralyMuller:2009}.}

	\appendix

	\section{Limiting covariance matrix of NSS-JD}\label{sec:appendix}
	In Appendix \ref{sec:appendix} we give the expression for the limiting covariance matrix $ \Sigma_{ \hat{W}_{X,K} } $ used in Theorem \ref{theo:limiting_normality}. The expression is based jointly on the results of Lemmas \ref{supp:lem:equi:gradient} and \ref{supp:lem:non:zero:hessian}, Theorems \ref{supp:thm:TCL}, \ref{supp:thm:joint:CLT:Z} and \ref{supp:thm:CLT:non:zero:mean} and Corollary \ref{supp:cor:CLT:X} given in the online supplementary material.
	
	In the following, let $\mathcal{S}_k$ be the set of $k \times k$ skew-symmetric matrices (for $M \in \mathcal{S}_k$, $M' = -M $) and let 
	\[
	\mathcal{U}_k = 
	\{ V = (V_{i,j})_{1 \leq i<j \leq k} ; V_{i,j} \in \mathbb{R} ~  \mbox{for} ~ 1 \leq i<j \leq k \}.
	\]
	We let $S: \mathcal{U}_k \to \mathcal{S}_k$ be defined, for $V \in \mathcal{U}_k$, as $S(V)_{i,i} = 0$, $S(V)_{i,j} = V_{i,j}$ and $S(V)_{j,i} = -V_{i,j}$ for $1 \leq i<j \leq k$. 
	
	Let $ \hat{\mathrm{cov}}_{Z,i} $ and $ \hat{\bar{\mathrm{cov}}}_{Z,K} $ be defined as the corresponding quantities in \eqref{eq:cov:x:1} and \eqref{eq:cov:x:2}, but with the series $ Z_t $ in place of $ X_t $. Let us further write $\hat{C}_i = \hat{\mathrm{cov}}_{Z,i} $, $ \hat{T} = 
	-
	(1/2)
	\left( 
	\hat{\bar{\mathrm{cov}}}_{Z,K}
	-
	I_p
	\right) $ and $ \hat{Q}_{i , jk} = \hat{C}_i
	e_j
	e_k'
	\hat{C}_i$ where $ e_j $ is the $ j $th standard basis vector of $ \mathbb{R}^p $. Define then, for $1 \leq  j < k \leq p$, the elements of $ \bar{\nabla}_0 \in \mathcal{U}_p $ as,
	\begin{align*}
		& [\bar{\nabla}_0]_{jk} \\
		= &
		-4
		\frac{1}{K}
		\sum_{i=1}^K
		e_k'
		\hat{Q}_{i , kk}
		e_j
		- 8
		e_k'
		\hat{T}
		\left(
		\frac{1}{K}
		\sum_{i=1}^K
		\mathbb{E}
		\left[
		\hat{Q}_{i , kk}
		e_j
		\right]
		\right)
		\\
		- &
		4
		e_k'
		\hat{T}
		\left(
		\frac{1}{K}
		\sum_{i=1}^K
		\mathbb{E}
		\left[
		\hat{Q}_{i , jk}
		e_k
		\right]
		\right)
		-
		4
		\left(
		\frac{1}{K}
		\sum_{i=1}^K
		\mathbb{E}
		\left[
		e_k'
		\hat{Q}_{i , kk}
		\right]
		\right)
		\hat{T}
		e_j
		\\
		+ &
		4
		\frac{1}{K}
		\sum_{i=1}^K
		e_j'
		\hat{Q}_{i , jj}
		e_k
		+ 8
		e_j'
		\hat{T}
		\left(
		\frac{1}{K}
		\sum_{i=1}^K
		\mathbb{E}
		\left[
		\hat{Q}_{i , jj}
		e_k
		\right]
		\right)
		\\
		+ &
		4
		e_j'
		\hat{T}
		\left(
		\frac{1}{K}
		\sum_{i=1}^K
		\mathbb{E}
		\left[
		\hat{Q}_{i , kj}
		e_j
		\right]
		\right)
		+ 4
		\left(
		\frac{1}{K}
		\sum_{i=1}^K
		\mathbb{E}
		\left[
		e_j'
		\hat{Q}_{i , jj}
		\right]
		\right)
		\hat{T}
		e_k. 
	\end{align*}
	Further, let $\Sigma_{\nabla}$ be the covariance matrix of $ K^{1/2} \bar{\nabla}_0$.
	
	Let $ C_i = \mathrm{cov}_{Z,i} $ from \eqref{eq:cov:z:1} and define, for $i=1, \ldots ,K$ and $a,b=1, \ldots, s$, $D_{Z,i}^{(a,b)}$ to be the $p \times p$ diagonal matrix with diagonal elements given as
	\[
	\left[
	D_{Z,i}^{(a,b)}
	\right]_{k,k}
	=
	\mathbb{E}
	\left[
	\left( Z_{\bl(i-1)+a}^{(k)} - \bar{Z}_i^{(k)} \right)
	\left( Z_{\bl(i-1)+b}^{(k)} - \bar{Z}_i^{(k)} \right)
	\right],
	\]
	where $ \bar{Z}_i^{(k)} = (1/s) \sum_{j=1}^s Z^{(k)}_{\bl(i-1)+j}$ is the $ k $th element of the mean vector $ \bar{Z}_i $ of the $ i $th block. Using the previous, define, for any $e,f=1,\ldots,p$, $e \neq f$, the quantity,
	\begin{align*}
		H_{e,f}
		=&
		\frac{4}{K}
		\sum_{i=1}^K
		\left(
		[C_i]_{e,e}
		-
		[C_i]_{f,f}
		\right)^2 \\
		+&
		\frac{8}{K}
		\sum_{i=1}^K
		\frac{1}{s^2} 
		\sum_{m,n=1}^s
		\left(
		\left[
		D_{Z,i}^{(m,n)}
		\right]_{e,e}
		-
		\left[
		D_{Z,i}^{(m,n)}
		\right]_{f,f}
		\right)^2.
	\end{align*}
	
	Let then $\Sigma_{\hat{U}}$ be the $p^2 \times p^2$, quadruple-indexed covariance matrix, with its $ (e,f),(g,h) $th element, $ e,f,g,h =1, \ldots ,p$, defined as,
	\begin{equation*}
		\begin{cases}
			0
			&
			~ ~ \mbox{if} ~ ~
			e=f ~ \mbox{or} ~ g=h
			\\
			\frac{1}{H_{e,f} H_{g,h}}
			[\Sigma_{\nabla}]_{(e,f),(g,h)}
			&
			~ ~ \mbox{if} ~ ~
			e < f \; , \; g < h
			\\
			\frac{1}{H_{e,f} H_{g,h}}
			\left( - [\Sigma_{\nabla}]_{(e,f),(h,g)} \right)
			&
			~ ~ \mbox{if} ~ ~
			e < f \; , \; g > h
			\\
			\frac{1}{H_{e,f} H_{g,h}}
			\left( - [\Sigma_{\nabla}]_{(f,e),(g,h)} \right)
			&
			~ ~ \mbox{if} ~ ~
			e > f \; , \; g < h
			\\
			\frac{1}{H_{e,f} H_{g,h}}
			[\Sigma_{\nabla}]_{(f,e),(h,g)}
			&
			~ ~ \mbox{if} ~ ~
			e > f \; , \; g > h.
		\end{cases}
	\end{equation*}
	
	Let then $\bar{\mathcal{E}}^{-1}$ be the linear transformation on $\mathcal{U}_p$ defined, for $1 \leq e < f \leq p$, by,
	
	\[
	(\bar{\mathcal{E}}^{-1}[V])_{e,f}
	=
	H_{e,f}^{-1} V_{e,f}.
	\]
	Construct the covariance matrix,
	\begin{align}\label{eq:w:z:k}
		\Sigma_{ \hat{W}_{Z,K} }
		=
		\Sigma_{ \hat{U} }
		+
		\frac{1}{4}
		\Sigma_{\hat{\bar{\mathrm{cov}}}_{Z,K}}
		-
		\Sigma_{\text{cross}},
	\end{align}
	where $\Sigma_{\hat{\bar{\mathrm{cov}}}_{Z,K}}$ is the $p^2 \times p^2$ covariance matrix of $ K^{1/2} (\hat{\bar{\mathrm{cov}}}_{Z,K} - I_p)$ and $ \Sigma_{\text{cross}} $ is the cross covariance matrix between $S
	(
	\bar{\mathcal{E}}^{-1}
	[
	\sqrt{K} \bar{\nabla}_0 
	]
	)$ and $ K^{1/2} (\hat{\bar{\mathrm{cov}}}_{Z,K} - I_p)$. Finally, the desired matrix $\Sigma_{\hat{W}_{X,K}}$ is then the $p^2 \times p^2$ covariance matrix of the random matrix 
	\[
	M A^{-1},
	\]
	where $M$ is a $p \times p$ random matrix with $p^2 \times p^2$ covariance matrix $\Sigma_{\hat{W}_{Z,K}}$ given in \eqref{eq:w:z:k}, and $ A^{-1} $ is the true unmixing matrix.

	\section*{Acknowledgment}
	
	The work of CM and KN was supported by the Austrian Science Fund P31881-N32. The work of JV was supported by Academy of Finland, Grant 335077.

	\newpage

	\section{Organization of the supplementary material}
	
	We first prove Theorems  \ref{theo:consistency} and \ref{theo:limiting_normality} in a simplified setting where the multivariate time series $Z$ is centered and where no centering is performed when computing the matrices $\hat{\mathrm{cov}}_{X,i}$. This simplified setting is described in Section \ref{supp:section:setting:notation}. Under this simplified setting, we first consider the case where the blind source separation procedure is applied directly to $Z$, so that the target unmixing matrix is $I_p$. This case is also described in Section \ref{supp:section:setting:notation}. Then, the consistency result (Theorem  \ref{theo:consistency}) is proved in Theorem \ref{supp:thm:consistency} and the asymptotic normality result (Theorem  \ref{theo:limiting_normality}) is proved in Theorem \ref{supp:thm:joint:CLT:Z}.
	
	These two results on $Z$ imply similar results on $X$, thanks to an equivariance property given in Lemma \ref{supp:lem:equivariance}. Thus, Corollary \ref{supp:cor:consistency:X} provides the consistency of the spatial blind source separation procedure to $A^{-1}$ (Theorem  \ref{theo:consistency}) and Corollary \ref{supp:cor:CLT:X} provides the asymptotic normality (Theorem  \ref{theo:limiting_normality}). In these two corollaries, the simplified setting of a zero mean and where no centering is performed is still considered.
	
	Finally, in Section \ref{supp:section:non-zero:mean}, we show that the results in this simplified setting actually imply the results in the setting described in the main body of the paper (non-zero means and centering when computing the matrices $\hat{\mathrm{cov}}_{X,i}$). This is stated in Theorem \ref{supp:thm:CLT:non:zero:mean}, that shows that the conclusions of Corollaries \ref{supp:cor:consistency:X} and \ref{supp:cor:CLT:X} still hold. Hence, Theorem \ref{supp:thm:CLT:non:zero:mean} and Corollaries \ref{supp:cor:consistency:X} and \ref{supp:cor:CLT:X} jointly provide the proofs of Theorems  \ref{theo:consistency} and \ref{theo:limiting_normality}.
	
	In this supplementaty material, we may repeat notation and conditions from the main body of the paper, for a more self contained and easier to read document.
	
	\section{Setting and notation} \label{supp:section:setting:notation}
	
	For a $r \times r$ matrix $M$, we let $\mathrm{diag}(M)$ be the matrix obtained by setting all the non-diagonal elements of $M$ to zero. For $i=1,...,r$, we let $M_i'$ be the $i$-th row of $M$. When $N$ is also a $r \times r$ matrix, we let $M \odot N$ be defined by $[ M \odot N]_{ij} = M_{ij} N_{ij}$ for $1 \leq i,j \leq r$.
	
	Let $||M||^2 = \sum_{a,b=1}^r M_{a,b}^2$. Let $\rho_{sup}(M)$ be the largest singular value of $M$.
	If $M$ is symmetric, we let $\lambda_{inf}(M)$ and $\lambda_{sup}(M)$ be its smallest and largest eigenvalues.
	If $M$ is symmetric non-negative definite, we let $M^{1/2}$ be the unique symmetric non-negative definite matrix $N$ satisfying $N^2 = M$.  We let $\mathcal{G}_r$ be the set of $r \times r$ matrices $G$ such that there exist a permutation $\sigma$ on $\{1,...,r\}$ and $s_1,...,s_r \in \{-1,1\}$ such that, for any $r \times 1$ vector $v$, for any $i=1,...,r$, $[Gv]_i = s_i v_{\sigma(i)}$.
	For a $r$-dimensional vector $v$, we let $||v||_l^l = \sum_{a=1}^r |v_a|^l$. We let $\# E$ denote the cardinality of a finite set $E$.  We let $e_a$ be the $a$-th base column vector of $\mathbb{R}^k$ for some $k \in \mathbb{N}$, where the value of $k$ will be clear from context. 
	
	For any $k \in \mathbb{N}$, let $\mathcal{M}_k$ be the set of $k \times k$ real matrices.
	We let $\mathcal{O}_k$ be the set of $k \times k$ real orthogonal matrices (for $M \in \mathcal{O}_k$, $M'M = I_k$). We let $\mathcal{S}_k$ be the set of $k \times k$ skew symmetric matrices (for $M \in \mathcal{S}_k$, $M' = -M $). We let 
	\[
	\mathcal{U}_k = 
	\{ V = (V_{i,j})_{1 \leq i<j \leq k} ; V_{i,j} \in \mathbb{R} ~  \mbox{for} ~ 1 \leq i<j \leq k \}.
	\]
	We let $S: \mathcal{U}_k \to \mathcal{S}_k$ be defined by, for $V \in \mathcal{U}_k$, $S(V) \in \mathcal{S}_k$ is defined by $S(V)_{i,i} = 0$, $S(V)_{i,j} = V_{i,j}$ and $S(V)_{j,i} = -V_{i,j}$ for $1 \leq i<j \leq k$. 
	Let $\exp: \mathcal{M}_k \to \mathcal{M}_k$ be the matrix exponential function (see e.g. Chapter 2 in \cite{hall2015lie}).

	In this supplementary material, we consider the case where 
	\[
	\mathbb{E}(Z_t) = 0 \text{   for all  } t \in \mathbb{N}.
	\]
	We provide the extension to non-zero means in Section \ref{supp:section:non-zero:mean}.
	
	Recall that $s \in \mathbb{N}$, $s>1$, is fixed. For $i =1,\ldots,K$, let
	\[
	\hat{\mathrm{cov}}_{Z,i}
	= \frac{1}{s} \sum_{j=1}^s  Z_{s(i-1)+j} Z_{s(i-1)+j}'
	\]
	and let
	\[
	\mathrm{cov}_{Z,i}
	= \frac{1}{s} \sum_{j=1}^s \mathbb{E} \left( Z_{s(i-1)+j} Z_{s(i-1)+j}' \right).
	\]
	We provide extensions to the case where the vectors $Z_{s(i-1)+j}$ are empirically centered in Section \ref{supp:section:non-zero:mean}. Recall that we let $K \to \infty$ for all the asymptotic results that are shown.
	Let 
	\[
	\hat{\bar{\mathrm{cov}}}_{Z,K} = \frac{1}{K} \sum_{i=1}^K \hat{\mathrm{cov}}_{Z,i}
	\]
	and
	\[
	\bar{\mathrm{cov}}_{Z,K} = \frac{1}{K} \sum_{i=1}^K \mathrm{cov}_{Z,i}.
	\]

	We let
	\begin{equation} \label{supp:eq:ref:hatUZK}
		\hat{U}_{Z,K} 
		\in
		\mathrm{argmax}_{ U \in \mathcal{O}_p }
		\sum_{i=1}^K
		||  
		\mathrm{diag}
		\left(
		U
		\hat{\bar{\mathrm{cov}}}_{Z,K}^{-1/2}
		\hat{\mathrm{cov}}_{Z,i}
		\hat{\bar{\mathrm{cov}}}_{Z,K}^{-1/2}
		U'
		\right)
		||^2.
	\end{equation}
	We let 
	\[
	\hat{W}_{Z,K} =
	\hat{U}_{Z,K} 
	\hat{\bar{\mathrm{cov}}}_{Z,K}^{-1/2}.
	\]
	We now recall 
	\begin{equation} \label{supp:eq:X:equal:A:Z}
		X_t = A Z_t
	\end{equation} 
	for $t \in \mathbb{N}$ where $A$ is a fixed invertible $p \times p$ matrix.
	
	Throughout this supplementary material, we assume that $\bar{\mathrm{cov}}_{Z,K} = I_p$, as is done in the main body of the paper.
	
	We now define $\hat{\mathrm{cov}}_{X,i}$, $\mathrm{cov}_{X,i}$, $\hat{\bar{\mathrm{cov}}}_{X,K}$, $\bar{\mathrm{cov}}_{X,K}$, $\hat{U}_{X,K}$, and $\hat{W}_{X,K}$ similarly as above, but where the multivariate Gaussian process $Z$ is replaced by the multivariate Gaussian process $X$. The selection of $\hat{U}_{X,K}$ in the set of maximizers is arbitrary.
	
	The next lemma provides an equivariance property that relates
	$\hat{W}_{X,K}$ to $\hat{W}_{Z,K}$.
	
	\begin{lem} \label{supp:lem:equivariance}
		For any choice of $\hat{U}_{X,K}$ and $\hat{W}_{X,K}$ such that \eqref{supp:eq:ref:hatUZK} holds (with $\hat{\bar{\mathrm{cov}}}_{Z,K}^{-1/2}$ and
		$\hat{\mathrm{cov}}_{Z,i}$
		replaced by 
		$\hat{\bar{\mathrm{cov}}}_{X,K}^{-1/2}$ and
		$\hat{\mathrm{cov}}_{X,i}$), there exists a choice of $\hat{U}_{Z,K}$ and $\hat{W}_{Z,K}$  such that \eqref{supp:eq:ref:hatUZK} holds and such that we have
		\[
		\hat{W}_{X,K} = \hat{W}_{Z,K} A^{-1}.
		\]
	\end{lem}
	
	\begin{proof}
		We have $\hat{\bar{\mathrm{cov}}}_{X,K} = A \hat{\bar{\mathrm{cov}}}_{Z,K} A'$. Hence, from Theorem 2.1 in \cite{ilmonen2012invariant}
		there exists a unique orthogonal matrix $\hat{V}$ such that 
		\[
		\hat{\bar{\mathrm{cov}}}_{X,K}^{-1/2}
		=
		\hat{V}\hat{\bar{\mathrm{cov}}}_{Z,K}^{-1/2} A^{-1}
		\]
		(remark that $\hat{\bar{\mathrm{cov}}}_{X,K}^{-1/2}$ is symmetric by definition).
		Then, we have
		\begin{align} 
			& \sum_{i=1}^K
			||  
			\mathrm{diag}
			\left(
			\hat{U}
			\hat{\bar{\mathrm{cov}}}_{X,K}^{-1/2}
			\hat{\mathrm{cov}}_{X,i}
			\hat{\bar{\mathrm{cov}}}_{X,K}^{-1/2}
			\hat{U}'
			\right)
			||^2
			\label{supp:eq:criterion:UXK}
			\\
			= &
			\sum_{i=1}^K
			||  
			\mathrm{diag}
			\left(
			\hat{U}
			\hat{V}\hat{\bar{\mathrm{cov}}}_{Z,K}^{-1/2} A^{-1}
			A \hat{\mathrm{cov}}_{Z,i} A'
			\hat{V}\hat{\bar{\mathrm{cov}}}_{Z,K}^{-1/2} A^{-1}
			\hat{U}'
			\right)
			||^2
			\notag
			\\
			= &
			\sum_{i=1}^K
			||  
			\mathrm{diag}
			\left(
			\hat{U}
			\hat{V}\hat{\bar{\mathrm{cov}}}_{Z,K}^{-1/2} 
			\hat{\mathrm{cov}}_{Z,i} A'
			[A^{-1}]'\hat{\bar{\mathrm{cov}}}_{Z,K}^{-1/2} \hat{V} '
			\hat{U}'
			\right)
			||^2
			\notag
			\\
			= &
			\sum_{i=1}^K
			||  
			\mathrm{diag}
			\left(
			\hat{U}
			\hat{V}\hat{\bar{\mathrm{cov}}}_{Z,K}^{-1/2} 
			\hat{\mathrm{cov}}_{Z,i} 
			\hat{\bar{\mathrm{cov}}}_{Z,K}^{-1/2} 
			\hat{V}' 
			\hat{U}'
			\right)
			||^2,
			\label{supp:eq:criterion:UZK}
		\end{align}
		where the second to last relation follows from the fact that
		$\hat{\bar{\mathrm{cov}}}_{X,K}^{-1/2}
		=
		\hat{V}\hat{\bar{\mathrm{cov}}}_{Z,K}^{-1/2} A^{-1}$
		is symmetric. 
		Thus, for any $\hat{U}_{X,K}$ maximizing \eqref{supp:eq:criterion:UXK}, the corresponding $\hat{U}_{Z,K} := \hat{U}_{X,K} \hat{V}$ satisfies \eqref{supp:eq:ref:hatUZK}.
		Furthermore 
		\begin{align*}
			\hat{W}_{X,K}
			= &
			\hat{U}_{Z,K} \hat{V}' 
			\hat{\bar{\mathrm{cov}}}_{X,K}^{-1/2}
			\\
			= &
			\hat{U}_{Z,K} \hat{V}' 
			\hat{V}\hat{\bar{\mathrm{cov}}}_{Z,K}^{-1/2} A^{-1}
			\\
			= &
			\hat{U}_{Z,K}\hat{\bar{\mathrm{cov}}}_{Z,K}^{-1/2} A^{-1}
			\\
			= &
			\hat{W}_{Z,K} A^{-1}.
		\end{align*}
	\end{proof}
	
	\section{Consistency}
	
	\begin{lem} \label{supp:lem:calcul:esp:cost}
		Let $U$ be an orthogonal $p \times p$ matrix and let $U_i'$ be its $i$-th row. Let, for $i=1,...,K$ and $a,b=1,...,s$, $D_{Z,i}^{(a,b)}$ be the $p \times p$ diagonal matrix defined by
		\[
		\left[
		D_{Z,i}^{(a,b)}
		\right]_{k,k}
		=
		\mathbb{E}
		\left(
		Z_{(i-1)s+a}^{(k)}
		Z_{(i-1)s+b}^{(k)}
		\right).
		\]
		We have
		\begin{align*}
			\sum_{i=1}^K
			\sum_{j=1}^p
			\mathbb{E}
			\left[
			\left(
			U_j'
			\hat{\mathrm{cov}}_{Z,i}
			U_j
			\right)^2
			\right]
			= &
			\sum_{i=1}^K
			\sum_{j=1}^p
			\left(
			U_j'
			\mathrm{cov}_{Z,i}
			U_j
			\right)^2
			+
			\frac{2}{s^2}
			\sum_{i=1}^K
			\sum_{a,b=1}^s
			\sum_{j=1}^p
			\left(
			U_j'
			D_{Z,i}^{(a,b)}
			U_j
			\right)^2.
		\end{align*}
	\end{lem}
	
	\begin{proof}
		We have
		\begin{align}
			\sum_{i=1}^K
			\sum_{j=1}^p
			\mathbb{E}
			\left[
			\left(
			U_j'
			\hat{\mathrm{cov}}_{Z,i}
			U_j
			\right)^2
			\right]
			= &
			\sum_{i=1}^K
			\sum_{j=1}^p
			\mathbb{E}
			\left[
			\left(
			\sum_{l,m = 1}^p
			(U_j)_l
			\left( \hat{\mathrm{cov}}_{Z,i} \right)_{l,m}
			(U_j)_m
			\right)^2
			\right]
			\label{supp:eq:mean:cost:start:un}
			\\
			= &
			\sum_{i=1}^K
			\sum_{j=1}^p
			\mathbb{E}
			\left[
			\sum_{l_1,m_1,l_2,m_2 = 1}^p
			(U_j)_{l_1}
			\left( \hat{\mathrm{cov}}_{Z,i} \right)_{l_1,m_1}
			(U_j)_{m_1}
			(U_j)_{l_2}
			\left( \hat{\mathrm{cov}}_{Z,i} \right)_{l_2,m_2}
			(U_j)_{m_2}
			\right].
			\label{supp:eq:mean:cost:start:deux}
		\end{align}
		We now compute, using Isserliss' theorem,
		\begin{align*}
			\mathbb{E}
			\left(
			\left( \hat{\mathrm{cov}}_{Z,i} \right)_{l_1,m_1}
			\left( \hat{\mathrm{cov}}_{Z,i} \right)_{l_2,m_2}
			\right)
			= &
			\frac{1}{s^2}
			\sum_{a,b = 1}^s
			\mathbb{E}
			\left( 
			Z_{(i-1)s+a}^{(l_1)}
			Z_{(i-1)s+a}^{(m_1)}
			Z_{(i-1)s+b}^{(l_2)}
			Z_{(i-1)s+b}^{(m_2)}
			\right)
			\\
			= &
			\frac{1}{s^2}
			\sum_{a,b = 1}^s
			\mathbb{E}
			\left( 
			Z_{(i-1)s+a}^{(l_1)}
			Z_{(i-1)s+a}^{(m_1)}
			\right)
			\mathbb{E}
			\left(
			Z_{(i-1)s+b}^{(l_2)}
			Z_{(i-1)s+b}^{(m_2)}
			\right)
			\\
			& +
			\frac{1}{s^2}
			\sum_{a,b = 1}^s
			\mathbb{E}
			\left( 
			Z_{(i-1)s+a}^{(l_1)}
			Z_{(i-1)s+b}^{(l_2)}
			\right)
			\mathbb{E}
			\left(
			Z_{(i-1)s+a}^{(m_1)}
			Z_{(i-1)s+b}^{(m_2)}
			\right)
			\\
			& +
			\frac{1}{s^2}
			\sum_{a,b = 1}^s
			\mathbb{E}
			\left( 
			Z_{(i-1)s+a}^{(l_1)}
			Z_{(i-1)s+b}^{(m_2)}
			\right)
			\mathbb{E}
			\left(
			Z_{(i-1)s+a}^{(m_1)}
			Z_{(i-1)s+b}^{(l_2)}
			\right).
		\end{align*}
		Then, we obtain, using the independence of the Gaussian processes $(Z^{(1)}_t)_{t \in \mathbb{N}},...,(Z^{(p)}_t)_{t \in \mathbb{N}}$,
		\begin{align} \label{supp:eq:moment:product:hat:cov}
			\mathbb{E}
			\left(
			\left( \hat{\mathrm{cov}}_{Z,i} \right)_{l_1,m_1}
			\left( \hat{\mathrm{cov}}_{Z,i} \right)_{l_2,m_2}
			\right)
			= &
			\left( \mathrm{cov}_{Z,i} \right)_{l_1,m_1}
			\left( \mathrm{cov}_{Z,i} \right)_{l_2,m_2}
			\\
			& +
			\mathbf{1}_{l_1=l_2}
			\mathbf{1}_{m_1=m_2}
			\frac{1}{s^2}
			\sum_{a,b = 1}^s
			\mathbb{E}
			\left( 
			Z_{(i-1)s+a}^{(l_1)}
			Z_{(i-1)s+b}^{(l_1)}
			\right)
			\mathbb{E}
			\left(
			Z_{(i-1)s+a}^{(m_1)}
			Z_{(i-1)s+b}^{(m_1)}
			\right)
			\notag
			\\
			& +
			\mathbf{1}_{l_1=m_2}
			\mathbf{1}_{m_1=l_2}
			\frac{1}{s^2}
			\sum_{a,b = 1}^s
			\mathbb{E}
			\left( 
			Z_{(i-1)s+a}^{(l_1)}
			Z_{(i-1)s+b}^{(l_1)}
			\right)
			\mathbb{E}
			\left(
			Z_{(i-1)s+a}^{(m_1)}
			Z_{(i-1)s+b}^{(m_1)}
			\right). \notag
		\end{align}
		Hence, from \eqref{supp:eq:mean:cost:start:deux}, we obtain
		\begin{flalign} \label{supp:eq:computation:mean:value:VjhatcovVj}
			&
			\sum_{i=1}^K
			\sum_{j=1}^p
			\mathbb{E}
			\left[
			\left(
			U_j'
			\hat{\mathrm{cov}}_{Z,i}
			U_j
			\right)^2
			\right]
			= & \notag
			\\
			&
			\sum_{i=1}^K
			\sum_{j=1}^p
			\sum_{l_1,m_1,l_2,m_2 = 1}^p
			(U_j)_{l_1}
			\left( \mathrm{cov}_{Z,i} \right)_{l_1,m_1}
			(U_j)_{m_1}
			(U_j)_{l_2}
			\left( \mathrm{cov}_{Z,i} \right)_{l_2,m_2}
			(U_j)_{m_2}
			& 
			\\
			&  +
			\sum_{i=1}^K
			\sum_{j=1}^p
			\sum_{l_1,m_1 = 1}^p
			(U_j)_{l_1}^2
			(U_j)_{m_1}^2
			\frac{1}{s^2}
			\sum_{a,b=1}^s
			\mathbb{E}
			\left( 
			Z_{(i-1)s+a}^{(l_1)}
			Z_{(i-1)s+b}^{(l_1)}
			\right)
			\mathbb{E}
			\left(
			Z_{(i-1)s+a}^{(m_1)}
			Z_{(i-1)s+b}^{(m_1)}
			\right)
			& \notag
			\\
			& +
			\sum_{i=1}^K
			\sum_{j=1}^p
			\sum_{l_1,m_1 = 1}^p
			(U_j)_{l_1}^2
			(U_j)_{m_1}^2
			\frac{1}{s^2}
			\sum_{a,b=1}^s
			\mathbb{E}
			\left( 
			Z_{(i-1)s+a}^{(l_1)}
			Z_{(i-1)s+b}^{(l_1)}
			\right)
			\mathbb{E}
			\left(
			Z_{(i-1)s+a}^{(m_1)}
			Z_{(i-1)s+b}^{(m_1)}
			\right). \notag
			&
		\end{flalign}
		In the above display, the triple sum in \eqref{supp:eq:computation:mean:value:VjhatcovVj} can be treated in the same (reverse) way as from \eqref{supp:eq:mean:cost:start:un} to \eqref{supp:eq:mean:cost:start:deux}. Hence, we obtain
		\begin{flalign*}
			&
			\sum_{i=1}^K
			\sum_{j=1}^p
			\mathbb{E}
			\left[
			\left(
			U_j'
			\hat{\mathrm{cov}}_{Z,i}
			U_j
			\right)^2
			\right]
			= &
			\\
			&
			\sum_{i=1}^K
			\sum_{j=1}^p
			\left(
			U_j'
			\mathrm{cov}_{Z,i}
			U_j
			\right)^2
			&
			\\
			&  +
			\frac{2}{s^2}
			\sum_{i=1}^K
			\sum_{j=1}^p
			\sum_{l_1,m_1 = 1}^p
			(U_j)_{l_1}^2
			(U_j)_{m_1}^2
			\sum_{a,b=1}^s
			\mathbb{E}
			\left( 
			Z_{(i-1)s+a}^{(l_1)}
			Z_{(i-1)s+b}^{(l_1)}
			\right)
			\mathbb{E}
			\left(
			Z_{(i-1)s+a}^{(m_1)}
			Z_{(i-1)s+b}^{(m_1)}
			\right)
			&
			\\
			& = 
			\sum_{i=1}^K
			\sum_{j=1}^p
			\left(
			U_j'
			\mathrm{cov}_{Z,i}
			U_j
			\right)^2
			&
			\\
			&
			+
			\frac{2}{s^2}
			\sum_{i=1}^K
			\sum_{j=1}^p
			\sum_{a,b=1}^s
			\left(
			\sum_{l_1 = 1}^p
			(U_j)_{l_1}^2
			\mathbb{E}
			\left( 
			Z_{(i-1)s+a}^{(l_1)}
			Z_{(i-1)s+b}^{(l_1)}
			\right)
			\right)
			\left(
			\sum_{m_1 = 1}^p
			(U_j)_{m_1}^2
			\mathbb{E}
			\left(
			Z_{(i-1)s+a}^{(m_1)}
			Z_{(i-1)s+b}^{(m_1)}
			\right)
			\right)
			\\
			& = 
			\sum_{i=1}^K
			\sum_{j=1}^p
			\left(
			U_j'
			\mathrm{cov}_{Z,i}
			U_j
			\right)^2
			&
			\\
			&
			+
			\frac{2}{s^2}
			\sum_{i=1}^K
			\sum_{a,b=1}^s
			\sum_{j=1}^p
			\left(
			\sum_{l_1 = 1}^p
			(U_j)_{l_1}^2
			\mathbb{E}
			\left( 
			Z_{(i-1)s+a}^{(l_1)}
			Z_{(i-1)s+b}^{(l_1)}
			\right)
			\right)^2.
		\end{flalign*}
		Thus,
		\begin{align*}
			\sum_{i=1}^K
			\sum_{j=1}^p
			\mathbb{E}
			\left[
			\left(
			U_j'
			\hat{\mathrm{cov}}_{Z,i}
			U_j
			\right)^2
			\right]
			= &
			\sum_{i=1}^K
			\sum_{j=1}^p
			\left(
			U_j'
			\mathrm{cov}_{Z,i}
			U_j
			\right)^2
			+
			\frac{2}{s^2}
			\sum_{i=1}^K
			\sum_{a,b=1}^s
			\sum_{j=1}^p
			\left(
			U_j'
			D_{Z,i}^{(a,b)}
			U_j
			\right)^2.
		\end{align*}
		Hence, the proof is concluded. 
	\end{proof}
	
	In this supplementary material, we let $C_{inf} >0$ and $0 < C_{sup} < + \infty$ denote generic constants (not depending on $K$) which may change from place to place. We restate Conditions \ref{cond:block:independence} and \ref{cond:max:variance:bis}, also using Condition \ref{cond:gaussian:independent}. 
	
	\begin{cond} \label{supp:cond:block:independence}
		There exists a fixed $L \in \mathbb{N}$ such that
		for any $i,j \in \{1,...,K\}$, $ |i - j| \geq L$, the Gaussian vectors $( Z_{(i-1)s+a}^{(k)} )_{a=1,...,s; k=1,...,p}$ and $( Z_{(j-1)s+a}^{(k)} )_{a=1,...,s; k=1,...,p}$ are independent.
	\end{cond}

	\begin{cond} \label{supp:cond:max:variance}
		We have
		\[
		\sup_{i \in \mathbb{N}}
		\max_{ j=1,...,p }
		\mathbb{E}
		\left(
		\left[
		Z_{i}^{(j)}
		\right]^2
		\right)
		\leq
		C_{sup}.
		\]
	\end{cond}
	
	\begin{lem} \label{supp:lem:basic:bounds}
		Assume that Conditions \ref{supp:cond:block:independence} and \ref{supp:cond:max:variance} hold.
		We have
		\begin{equation} \label{supp:eq:lem:bound:un}
			\max_{i=1,...,K}
			\rho_{sup}( \mathrm{cov}_{Z,i} )
			\leq C_{sup},
		\end{equation}
		and for any $r \in \mathbb{N}$,
		\begin{equation} \label{supp:eq:lem:bound:deux}
			\max_{i=1,...,K}
			\mathbb{E}
			\left(
			\rho_{sup}( \hat{\mathrm{cov}}_{Z,i} )^r
			\right)
			\leq C_{sup}
		\end{equation}
		and
		\begin{equation} \label{supp:eq:lem:bound:trois}
			\left(
			\mathbb{E}
			\left[
			\rho_{sup}
			\left(
			\hat{\bar{\mathrm{cov}}}_{Z,K}
			-
			I_p
			\right)^r
			\right]
			\right)^{1/r}
			\leq C_{sup}
			\frac{1}{\sqrt{K}}.
		\end{equation}
		
	\end{lem}
	
	\begin{proof}
		The matrix $\mathrm{cov}_{Z,i}$ is diagonal, and we have for $k=1,...,p$,
		\[
		\left|
		\left[ \mathrm{cov}_{Z,i} \right]_{k,k}
		\right|
		=
		\frac{1}{s}
		\sum_{a=1}^s
		\mathbb{E}
		\left(
		\left[
		Z_{(i-1)s+a}^{(k)}
		\right]^2
		\right)
		\leq C_{sup},
		\]
		from Condition \ref{supp:cond:max:variance}. Hence, \eqref{supp:eq:lem:bound:un} holds by equivalence of norms in fixed dimension $p$. Again by equivalence of norms we obtain for $i=1,...,K$
		\begin{align*}
			\mathbb{E}
			\left(
			\rho_{sup}( \hat{\mathrm{cov}}_{Z,i} )^r
			\right)
			\leq &
			C_{sup}
			\mathbb{E}
			\left(
			\left[
			\sum_{k,l=1}^p
			\sum_{a=1}^s
			\left|
			Z_{(i-1)s+a}^{(k)}
			Z_{(i-1)s+a}^{(l)}
			\right|
			\right]^r
			\right)
			\\
			= &
			C_{sup}
			\sum_{
				\substack{
					(k_1,l_1,a_1) \in \{1,...,p\}^2 \times \{1,...,s\} 
					\\
					\vdots
					\\
					(k_r,l_r,a_r) \in \{1,...,p\}^2 \times \{1,...,s\} 
				}
			}
			\mathbb{E}
			\left(
			\left|
			Z_{(i-1)s+a_1}^{(k_1)}
			Z_{(i-1)s+a_1}^{(l_1)}
			...
			Z_{(i-1)s+a_r}^{(k_r)}
			Z_{(i-1)s+a_r}^{(l_r)}
			\right|
			\right)
			\\
			\leq & C_{sup},
		\end{align*}
		since $p$ and $s$ are fixed, from Condition \ref{supp:cond:max:variance}, from Theorem 1 in \cite{li2012gaussian} and from the Cauchy-Schwarz inequality.
		
		Let us turn to \eqref{supp:eq:lem:bound:trois}.
		We have
		\[
		\rho_{sup}
		\left(
		\hat{\bar{\mathrm{cov}}}_{Z,K}
		-
		I_p
		\right)^r
		\leq
		C_{sup}
		\sum_{k,l=1}^p
		\left|
		\left[  \hat{\bar{\mathrm{cov}}}_{Z,K} \right]_{k,l}
		-
		\left[ I_p \right]_{k,l}
		\right|^r.
		\]
		Then, for any $k,l \in \{1,...,p\}$, is is sufficient to show that
		\begin{equation} \label{supp:eq:to:show:moment:hatbarZ}
			\mathbb{E}
			\left(
			\left|
			\left[  \hat{\bar{\mathrm{cov}}}_{Z,K} \right]_{k,l}
			-
			\left[ I_p \right]_{k,l}
			\right|^r
			\right)
			\leq 
			\frac{C_{sup}}{ K^{r/2} }.
		\end{equation}
		This is true for $r=2$, since then $\left[  \hat{\bar{\mathrm{cov}}}_{Z,K} \right]_{k,l}
		-
		\left[ I_p \right]_{k,l}$ is of the form
		\begin{equation} \label{supp:eq:generic:sum:one}
			\frac{1}{K}
			\sum_{i=1}^K
			a_i,
		\end{equation}
		where the $a_i$ are centered random variables with bounded variances and where $a_i$ and $a_j$ are independent for $|i - j| \geq L$. Thus the mean value of the square of \eqref{supp:eq:generic:sum:one} is of order $O(1/K)$ as $K \to \infty$.
		Hence, \eqref{supp:eq:to:show:moment:hatbarZ} also holds for $r=1$. We have, using $|t_1 + \cdots + t_L|^r \leq L^r |t_1|^r + \cdots + L^r |t_L|^r$ for $t_1 , \ldots , t_L \in \mathbb{R}$, and letting $i \text{ mod } L$ be the remainder of the Euclidean division of $i$ by $L$,
		\begin{align*}
			\mathbb{E}
			\left(
			\left|
			\left[  \hat{\bar{\mathrm{cov}}}_{Z,K} \right]_{k,l}
			-
			\left[ I_p \right]_{k,l}
			\right|^r
			\right)
			= &
			\frac{1}{K^{r}}
			\mathbb{E}
			\left(
			\left|
			\sum_{i=1}^K
			\left[
			\hat{\mathrm{cov}}_{Z,i}
			-
			\mathrm{cov}_{Z,i}
			\right]_{k,l}
			\right|^r
			\right)
			\\
			\leq &
			\frac{L^r}{K^{r}}
			\sum_{a = 0}^{L-1}
			\mathbb{E}
			\left(
			\left|
			\sum_{\substack{i=1 \\ i \text{ mod } L = a}}^K
			\left[
			\hat{\mathrm{cov}}_{Z,i}
			-
			\mathrm{cov}_{Z,i}
			\right]_{k,l}
			\right|^r
			\right)
			\\
			\leq &
			L
			\frac{L^r}{K^{r}}
			C_{sup}
			\max( K , K^{r/2} ),
		\end{align*}
		because in each of the $L$ inner sums above, the summands are independent,
		from the Rosenthal inequality \cite{rosenthal1970subspaces} and from \eqref{supp:eq:lem:bound:un} and \eqref{supp:eq:lem:bound:deux}. Thus the proof is concluded.
	\end{proof}
	
	\begin{lem} \label{supp:lem:removing:hat}
		Assume that conditions \ref{supp:cond:block:independence} and \ref{supp:cond:max:variance} hold. Then, we have, for any fixed $p \times p$ orthogonal matrix $U$, with rows $U_1',\ldots,U_p'$,
		\[
		\sum_{i=1}^K
		\sum_{j=1}^p
		\left(
		U_j'
		\hat{\mathrm{cov}}_{Z,i}
		U_j
		\right)^2
		-
		\sum_{i=1}^K
		\sum_{j=1}^p
		\left(
		U_j'
		\hat{\bar{\mathrm{cov}}}_{Z,K}^{-1/2}
		\hat{\mathrm{cov}}_{Z,i}
		\hat{\bar{\mathrm{cov}}}_{Z,K}^{-1/2}
		U_j
		\right)^2
		=
		o_p(K).
		\]
	\end{lem}
	
	\begin{proof}
		We have
		\begin{flalign*}
			&
			\left|
			\sum_{i=1}^K
			\sum_{j=1}^p
			\left(
			U_j'
			\hat{\mathrm{cov}}_{Z,i}
			U_j
			\right)^2
			-
			\sum_{i=1}^K
			\sum_{j=1}^p
			\left(
			U_j'
			\hat{\bar{\mathrm{cov}}}_{Z,K}^{-1/2}
			\hat{\mathrm{cov}}_{Z,i}
			\hat{\bar{\mathrm{cov}}}_{Z,K}^{-1/2}
			U_j
			\right)^2
			\right|
			&
			\\
			& 
			=
			\Big|
			\sum_{i=1}^K
			\sum_{j=1}^p
			\left(
			\left[
			U_j'
			\hat{\mathrm{cov}}_{Z,i}
			U_j
			\right]
			+
			\left[
			U_j'
			\hat{\bar{\mathrm{cov}}}_{Z,K}^{-1/2}
			\hat{\mathrm{cov}}_{Z,i}
			\hat{\bar{\mathrm{cov}}}_{Z,K}^{-1/2}
			U_j
			\right]
			\right)
			\\
			&
			\left[
			U_j'
			\left\{
			\hat{\bar{\mathrm{cov}}}_{Z,K}^{-1/2}
			\hat{\mathrm{cov}}_{Z,i}
			\hat{\bar{\mathrm{cov}}}_{Z,K}^{-1/2}
			-
			\hat{\mathrm{cov}}_{Z,i}
			\right\}
			U_j
			\right]
			\Big| 
			\\
			&
			=
			\sum_{i=1}^K
			\sum_{j=1}^p
			a_{i,j}
			b_{i,j},
		\end{flalign*}
		say.
		Let $0 < \epsilon < 1$. We have from \eqref{supp:eq:lem:bound:trois} that 
		$P( E_{\epsilon,K} ) \to 1$ as $K \to \infty$, where $E_{\epsilon,K}$ is the event $\lambda_{inf}( \hat{\bar{\mathrm{cov}}}_{Z,K} ) \geq \epsilon$. 
		
		We have, under the event $E_{\epsilon,K}$, 
		\begin{align*}
			\max_{ \substack{i=1,...,K \\ j=1,...,p }}
			| a_{i,j}|
			\leq &
			\max_{ \substack{i=1,...,K}}
			\rho_{sup}(  \hat{\mathrm{cov}}_{Z,i}  )
			+
			\max_{ \substack{i=1,...,K}}
			\frac{1}{\epsilon} 
			\rho_{sup}(  \hat{\mathrm{cov}}_{Z,i}  ).
		\end{align*}
		Then, since $p$ and $s$ are fixed, from the Cauchy-Schwarz inequality and from Condition \ref{supp:cond:max:variance}, we obtain
		\begin{align*}
			\max_{ \substack{i=1,...,K}}
			\rho_{sup}(  \hat{\mathrm{cov}}_{Z,i}  )
			\leq &
			C_{sup}
			\max_{ \substack{i=1,...,K \\ a=1,...,s \\ k=1,...,p}}
			\left[
			Z_{(i-1)s+a}^{(k)}
			\right]^2
			\\
			\leq &
			C_{sup}
			\max_{ \substack{i=1,...,K \\ a=1,...,s \\ k=1,...,p}}
			\frac{1}{
				\mathbb{E} \left(
				\left[
				Z_{(i-1)s+a}^{(k)}
				\right]^2
				\right)
			}
			\left[
			Z_{(i-1)s+a}^{(k)}
			\right]^2
			\\
			= &
			O_p( \log(K) )
		\end{align*}
		from Equation A.3 in \cite{chatterjee2014superconcentration}. Hence, since $P( E_{\epsilon,K} ) \to 1$ as $K \to \infty$, we obtain
		\begin{equation} \label{supp:eq:max:a:ologK}
			\max_{ \substack{i=1,...,K \\ j=1,...,p }}
			| a_{i,j}|
			=
			O_p(\log(K)).
		\end{equation}
		We have
		\begin{align*}
			\sum_{i=1}^K
			\sum_{j=1}^p
			\left|
			b_{i,j}
			\right|
			= &
			\sum_{i=1}^K
			\sum_{j=1}^p
			\left|
			U_j'
			\left\{
			\hat{\bar{\mathrm{cov}}}_{Z,K}^{-1/2}
			\hat{\mathrm{cov}}_{Z,i}
			\hat{\bar{\mathrm{cov}}}_{Z,K}^{-1/2}
			-
			\hat{\mathrm{cov}}_{Z,i}
			\right\}
			U_j
			\right|
			\\
			\leq &
			\sum_{i=1}^K
			\sum_{j=1}^p
			\rho_{sup}
			\left(
			\hat{\bar{\mathrm{cov}}}_{Z,K}^{-1/2}
			\hat{\mathrm{cov}}_{Z,i}
			\hat{\bar{\mathrm{cov}}}_{Z,K}^{-1/2}
			-
			\hat{\mathrm{cov}}_{Z,i}
			\right)
			\\
			\leq &
			\sum_{i=1}^K
			\sum_{j=1}^p
			\rho_{sup}
			\left(
			\hat{\bar{\mathrm{cov}}}_{Z,K}^{-1/2}
			\hat{\mathrm{cov}}_{Z,i}
			\left[
			\hat{\bar{\mathrm{cov}}}_{Z,K}^{-1/2}
			-
			I_p
			\right]
			\right)
			\\
			& +
			\sum_{i=1}^K
			\sum_{j=1}^p
			\rho_{sup}
			\left(
			\left[
			\hat{\bar{\mathrm{cov}}}_{Z,K}^{-1/2}
			-
			I_p
			\right]
			\hat{\mathrm{cov}}_{Z,i}
			\right)
			\\
			\leq &
			\rho_{sup}
			\left(
			\hat{\bar{\mathrm{cov}}}_{Z,K}^{-1/2}
			-
			I_p
			\right)
			\sum_{i=1}^K
			\sum_{j=1}^p
			\rho_{sup}
			\left(
			\hat{\mathrm{cov}}_{Z,i}
			\right)
			\rho_{sup}
			\left(
			\hat{\bar{\mathrm{cov}}}_{Z,K}^{-1/2}
			\right)
			\\
			& +
			\rho_{sup}
			\left(
			\hat{\bar{\mathrm{cov}}}_{Z,K}^{-1/2}
			-
			I_p
			\right)
			\sum_{i=1}^K
			\sum_{j=1}^p
			\rho_{sup}
			\left(
			\hat{\mathrm{cov}}_{Z,i}
			\right).
		\end{align*}
		In the above display, the two last sums are of order $O_p(K)$, from \eqref{supp:eq:lem:bound:deux} and from the fact that $P( E_{\epsilon,K} ) \to 1$ as $K \to \infty$. Furthermore
		\begin{align*}
			\rho_{sup}
			\left(
			\hat{\bar{\mathrm{cov}}}_{Z,K}^{-1/2}
			-
			I_p
			\right)
			\leq &
			\rho_{sup}
			\left(
			\hat{\bar{\mathrm{cov}}}_{Z,K}^{-1/2}
			\right)
			\rho_{sup}
			\left(
			\hat{\bar{\mathrm{cov}}}_{Z,K}^{1/2}
			-
			I_p
			\right)
			\\
			= &
			O_p(1)
			\rho_{sup}
			\left(
			\hat{\bar{\mathrm{cov}}}_{Z,K}^{1/2}
			-
			I_p
			\right).
		\end{align*}
		On the event $E_{\epsilon,K}$, it is well known that there exits a finite constant $C_{sup,\epsilon}$ such that
		\[
		\rho_{sup}
		\left(
		\hat{\bar{\mathrm{cov}}}_{Z,K}^{1/2}
		-
		I_p
		\right)
		\leq 
		C_{sup,\epsilon}
		\rho_{sup}
		\left(
		\hat{\bar{\mathrm{cov}}}_{Z,K}
		-
		I_p
		\right).
		\]
		Hence, from \eqref{supp:eq:lem:bound:trois}, we obtain
		\[
		\sum_{i=1}^K
		\sum_{j=1}^p
		\left|
		b_{i,j}
		\right|
		=
		O_p( \sqrt{K} ).
		\]
		Hence, from \eqref{supp:eq:max:a:ologK}, the proof is concluded.
	\end{proof}
	
	\begin{lem} \label{supp:lem:law:large:number}
		Assume that conditions \ref{supp:cond:block:independence} and \ref{supp:cond:max:variance} hold. Then, we have, for any fixed $p \times p$ orthogonal matrix $U$, with rows $U_1',\ldots,U_p'$,
		\[
		\sum_{i=1}^K
		\sum_{j=1}^p
		\left(
		U_j'
		\hat{\mathrm{cov}}_{Z,i}
		U_j
		\right)^2
		-
		\sum_{i=1}^K
		\sum_{j=1}^p
		\mathbb{E}
		\left(
		\left(
		U_j'
		\hat{\mathrm{cov}}_{Z,i}
		U_j
		\right)^2
		\right)
		=
		o_p(K).
		\]
	\end{lem}
	\begin{proof}
		Since the quantity to bound (in absolute value) is of the form
		\begin{equation} \label{supp:eq:generic:sum:two}
			\sum_{i=1}^K
			a_i,
		\end{equation}
		where the $a_i$ are centered random variables where $a_i$ and $a_j$ are independent for $|i - j| \geq L$, it is sufficient to show that
		\[
		\max_{i=1,...,K}
		\max_{j=1,...,p}
		\mathbb{E}
		\left(
		\left(
		U_j'
		\hat{\mathrm{cov}}_{Z,i}
		U_j
		\right)^4
		\right)
		=O(1)
		\]
		as $K \to \infty$. We have
		\[
		\mathbb{E}
		\left(
		\left(
		U_j'
		\hat{\mathrm{cov}}_{Z,i}
		U_j
		\right)^4
		\right)
		\leq
		\mathbb{E}
		\left(
		\rho_{sup}
		\left(
		\hat{\mathrm{cov}}_{Z,i}
		\right)^4
		\right),
		\]
		so the proof is concluded because of \eqref{supp:eq:lem:bound:deux}.
	\end{proof}
	
	\begin{lem}\label{supp:lem:sup:bound}
		Assume that conditions \ref{supp:cond:block:independence} and \ref{supp:cond:max:variance} hold.
		Recall that $\mathcal{O}_p$ is the set of $p \times p$ orthogonal matrices. For $U \in \mathcal{O}_p$ we let $U_j'$ be the $j$-th row of $U$. We have
		\begin{flalign*}
			&
			\sup_{ U \in \mathcal{O}_p }
			\left|
			\sum_{i=1}^K
			\sum_{j=1}^p
			\left(
			U_j'
			\hat{\bar{\mathrm{cov}}}_{Z,K}^{-1/2}
			\hat{\mathrm{cov}}_{Z,i}
			\hat{\bar{\mathrm{cov}}}_{Z,K}^{-1/2}
			U_j
			\right)^2
			-
			\sum_{i=1}^K
			\sum_{j=1}^p
			\left(
			U_j'
			\mathrm{cov}_{Z,i}
			U_j
			\right)^2
			\right.
			&
			\\
			& -
			\left.
			\frac{2}{s^2}
			\sum_{i=1}^K
			\sum_{a,b=1}^s
			\sum_{j=1}^p
			\left(
			U_j'
			D_{Z,i}^{(a,b)}
			U_j
			\right)^2
			\right| 
			&
			\\
			&
			=
			o_p(K).
			&
		\end{flalign*}
	\end{lem}
	
	\begin{proof}
		
		From Lemmas \ref{supp:lem:calcul:esp:cost}, \ref{supp:lem:removing:hat} and \ref{supp:lem:law:large:number}, we have, for any fixed $U \in  \mathcal{O}_p$ that
		\begin{align*}
			&
			\left|
			\sum_{i=1}^K
			\sum_{j=1}^p
			\left(
			U_j'
			\hat{\bar{\mathrm{cov}}}_{Z,K}^{-1/2}
			\hat{\mathrm{cov}}_{Z,i}
			\hat{\bar{\mathrm{cov}}}_{Z,K}^{-1/2}
			U_j
			\right)^2
			-
			\sum_{i=1}^K
			\sum_{j=1}^p
			\left(
			U_j'
			\mathrm{cov}_{Z,i}
			U_j
			\right)^2
			\right.
			&
			\\
			& -
			\left.
			\frac{2}{s^2}
			\sum_{i=1}^K
			\sum_{a,b=1}^s
			\sum_{j=1}^p
			\left(
			U_j'
			D_{Z,i}^{(a,b)}
			U_j
			\right)^2
			\right| 
			&
			\\
			&
			=
			o_p(K).
			&
		\end{align*}
		
		Hence, because $ \mathcal{O}_p$ is compact, it is sufficient to show that, letting $\nabla_{U_j} \left[ f(U_j) \right]$ denote the gradient of a function $f: \mathbb{R}^p \to \mathbb{R}$ evaluated at $U_j \in \mathbb{R}^p$, we have for $j=1,...,p$,
		\[
		\sup_{|| U_j || = 1}
		\left| \left| 
		\sum_{i=1}^K
		\nabla_{U_j}
		\left[
		\left(
		U_j'
		\hat{\bar{\mathrm{cov}}}_{Z,K}^{-1/2}
		\hat{\mathrm{cov}}_{Z,i}
		\hat{\bar{\mathrm{cov}}}_{Z,K}^{-1/2}
		U_j
		\right)^2
		\right]
		\right|
		\right|
		=
		O_p(K),
		\]
		\[
		\sup_{|| U_j || = 1}
		\left| \left|
		\sum_{i=1}^K
		\nabla_{U_j}
		\left[
		\left(
		U_j'
		\mathrm{cov}_{Z,i}
		U_j
		\right)^2
		\right]
		\right|
		\right|
		=
		O(K)
		\]
		and
		\[
		\sup_{|| U_j || = 1}
		\left|
		\left|
		\frac{2}{s^2}
		\sum_{i=1}^K
		\sum_{a,b=1}^s
		\nabla_{U_j}
		\left[
		\left(
		U_j'
		D_{Z,i}^{(a,b)}
		U_j
		\right)^2
		\right] 
		\right|
		\right|
		=
		O(K).
		\]
		From \eqref{supp:eq:lem:bound:deux} and \eqref{supp:eq:lem:bound:trois},
		in order to prove the three above displays, it is sufficient to show that for a sequence $(M_j)_{j \in \mathbb{N}}$ of random symmetric $p \times p$ matrices and for a random matrix $N$ satisfying
		\[
		\max_{i \in \mathbb{N}}
		\mathbb{E} 
		\left[
		\left(
		\rho_{sup}(M_i)
		\right)^2
		\right]
		\leq
		C_{sup}
		\]
		and
		\[
		\rho_{sup}(N) = O_p(1),
		\]
		we have
		\[
		\sup_{|| U_j || = 1}
		\left| \left|
		\sum_{i=1}^K
		\nabla_{U_j}
		\left[
		\left(
		U_j'
		N
		M_i
		N
		U_j
		\right)^2
		\right]
		\right| \right|
		=
		O_p(K).
		\]
		We have
		\begin{align*}
			\sup_{|| U_j || = 1}
			\left|
			\left|
			\sum_{i=1}^K
			\nabla_{U_j}
			\left[
			\left(
			U_j'
			N
			M_i
			N
			U_j
			\right)^2
			\right]
			\right|
			\right|
			=
			&
			\sup_{|| U_j || = 1}
			\left|
			\left|
			\sum_{i=1}^K
			4
			\left(
			U_j'
			N
			M_i
			N
			U_j
			\right)
			N
			M_i
			N
			U_j
			\right|
			\right|
			\\
			\leq
			&
			C_{sup} O_p(1)
			\sum_{i=1}^K
			\rho_{sup}( M_i )^2,
		\end{align*}
		so the proof is concluded.
	\end{proof}
	
	The next condition is a restatement of Condition \ref{cond:asymptotically:distinct:eigenvalues}.
	\begin{cond} \label{supp:cond:asymptotically:distinct:eigenvalues}
		There exists a strictly increasing sequence $(i_k)_{k \in \mathbb{N}}$, such that $i_k \in \mathbb{N}$ for all $k \in \mathbb{N}$ and such that, with $N_K = \# \{ k =1,...,K ; i_k \leq K \}$,
		we have $\liminf  N_K /K >0$ as $K \to \infty$. There exists $\delta >0$, such that
		\[
		\inf_{k \in \mathbb{N}}
		\min_{ \substack{ i,j=1,...,p \\ i \neq j } }
		\left|
		\left [\mathrm{cov}_{Z,i_k} \right]_{i,i}
		-
		\left [\mathrm{cov}_{Z,i_k} \right]_{j,j}
		\right|
		\geq \delta
		\]
		and 
		\[
		\inf_{k \in \mathbb{N}}
		\min_{ \substack{ i=1,...,p  } }
		\left [\mathrm{cov}_{Z,i_k} \right]_{i,i}
		\geq \delta.
		\]
	\end{cond}

	\begin{thm}\label{supp:thm:consistency}
		Assume that conditions  \ref{supp:cond:block:independence}, \ref{supp:cond:max:variance} and \ref{supp:cond:asymptotically:distinct:eigenvalues} hold.
		Then for any sequence $\hat{U}_{Z,K}$ in \eqref{supp:eq:ref:hatUZK}, there exists a sequence $\hat{G}_K \in \mathcal{G}_p$ such that
		\[
		\hat{G}_K \hat{U}_{Z,K}
		\overset{p}{\underset{K \to \infty}{\to}}
		I_p
		\] 
		and 
		\[
		\hat{G}_K \hat{W}_{Z,K}
		\overset{p}{\underset{K \to \infty}{\to}}
		I_p.
		\] 
	\end{thm}
	
	\begin{proof}
		
		Let $\mathcal{U}_0$ be the set of $p \times p$ orthogonal matrices $U$, with rows $U_1',\ldots,U_p'$, satisfying
		\[
		\mbox{for all $i=1,...,p$}
		~ ~
		\sum_{k=1,...,p} [U_i]_k \geq 0
		~ ~
		\mbox{and the sequence}
		~ ~
		\left( \sum_{j=1}^p j [U_i]^2_j  \right)_{i=1,...,p}
		~ ~
		\mbox{is ascending.}
		\]
		
		Consider a sequence $\hat{U}_{Z,K}$ in \eqref{supp:eq:ref:hatUZK}. Then, there exists a sequence $\hat{G}_K \in \mathcal{G}_p$ such that 
		\[
		\hat{G}_K \hat{U}_{Z,K} 
		\in
		\mathrm{argmax}_{ U \in \mathcal{U}_0 }
		\sum_{i=1}^K
		\left| \left|
		\mathrm{diag}
		\left(
		U
		\hat{\bar{\mathrm{cov}}}_{Z,K}^{-1/2}
		\hat{\mathrm{cov}}_{Z,i}
		\hat{\bar{\mathrm{cov}}}_{Z,K}^{-1/2}
		U'
		\right)
		\right| \right|^2.
		\]
		We use the following shorthand, for $U \in \mathcal{O}_p$ with rows $U_1',\ldots,U_p'$:
		\begin{align*}
			\hat{\Delta}_K(U) &= \frac{1}{K}
			\sum_{i=1}^K
			\sum_{j=1}^p
			\left(
			U_j'
			\hat{\bar{\mathrm{cov}}}_{Z,K}^{-1/2}
			\hat{\mathrm{cov}}_{Z,i}
			\hat{\bar{\mathrm{cov}}}_{Z,K}^{-1/2}
			U_j
			\right)^2, \\
			D_K(U) &= \frac{1}{K}
			\sum_{i=1}^K
			\sum_{j=1}^p
			\left(
			U_j'
			\mathrm{cov}_{Z,i}
			U_j
			\right)^2
			+
			\frac{2}{K s^2}
			\sum_{i=1}^K
			\sum_{a,b=1}^s
			\sum_{j=1}^p
			\left(
			U_j'
			D_{Z,i}^{(a,b)}
			U_j
			\right)^2.
		\end{align*}
		The statement of Lemma \ref{supp:lem:sup:bound} can now be expressed as
		\begin{equation} \label{supp:eq:unif:conv:hat:Delta}
			\sup_{ U \in \mathcal{O}_p }
			\left| \hat{\Delta}_K(U) - D_K(U) \right| = o_p(1).
		\end{equation}
		Let $\epsilon >0$.
		For any $U \in \mathcal{U}_0$, with rows $U_1',\ldots,U_p'$, such that $|| U - I_p || \geq \epsilon$, we have, with $e_j$ the $j-th$ basis column vector of $\mathbb{R}^p$,
		\begin{flalign*}
			& D_K(I_p) - D_K(U)
			= &
			\\
			&
			\frac{1}{K}
			\sum_{i=1}^K
			\sum_{j=1}^p
			\left(
			e_j'
			\mathrm{cov}_{Z,i}
			e_j
			\right)^2
			-
			\frac{1}{K}
			\sum_{i=1}^K
			\sum_{j=1}^p
			\left(
			U_j'
			\mathrm{cov}_{Z,i}
			U_j
			\right)^2
			\\
			&
			+
			\frac{2}{K s^2}
			\sum_{i=1}^K
			\sum_{a,b=1}^s
			\sum_{j=1}^p
			\left(
			e_j'
			D_{Z,i}^{(a,b)}
			e_j
			\right)^2
			-
			\frac{2}{K s^2}
			\sum_{i=1}^K
			\sum_{a,b=1}^s
			\sum_{j=1}^p
			\left(
			U_j'
			D_{Z,i}^{(a,b)}
			U_j
			\right)^2.
			&
		\end{flalign*}
		Since the matrices $\mathrm{cov}_{Z,i}$ and $D_{Z,i}^{(a,b)}$ are diagonal, we obtain, with the notation of Condition \ref{supp:cond:asymptotically:distinct:eigenvalues},
		\begin{align*}
			D_K(I_p) - D_K(U)
			\geq 
			\frac{1}{K}
			\sum_{ \substack{ k=1,...,K \\ i_k \leq K}}
			\sum_{j=1}^p
			\left(
			e_j'
			\mathrm{cov}_{Z,i_k}
			e_j
			\right)^2
			-
			\frac{1}{K}
			\sum_{ \substack{ k=1,...,K \\ i_k \leq K}}
			\sum_{j=1}^p
			\left(
			U_j'
			\mathrm{cov}_{Z,i_k}
			U_j
			\right)^2.
		\end{align*}

		Then, by compacity, from Condition \ref{supp:cond:asymptotically:distinct:eigenvalues} and from Theorem 3 in \cite{belouchrani1997blind}, we can show that
		\[
		\inf_{k \in \mathbb{N}}
		\inf_{\substack{ U \in \mathcal{U}_0 \\ || U - I_p || \geq \epsilon \\  U ~\text{has rows} ~ U_1',\ldots,U_p' }}
		\left(
		\sum_{j=1}^p
		\left(
		e_j'
		\mathrm{cov}_{Z,i_k}
		e_j
		\right)^2
		-
		\sum_{j=1}^p
		\left(
		U_j'
		\mathrm{cov}_{Z,i_k}
		U_j
		\right)^2
		\right)
		\geq C_{inf}.
		\]
		Hence we obtain from Condition \ref{supp:cond:asymptotically:distinct:eigenvalues}
		\begin{equation} \label{supp:eq:unique:maximum:DK}
			\inf_{\substack{ U \in \mathcal{U}_0 \\ || U - I_p || \geq \epsilon }}
			D_K(I_p) - D_K(U)
			\geq 
			\frac{
				N_K C_{inf}
			}
			{
				K
			}
			\geq C_{inf}.
		\end{equation}
		We now have
		\begin{align*}
			P(\| \hat{G}_K \hat{U}_{Z,K}  - I_p \| \geq \epsilon) 
			\leq &
			P \left(
			\sup_{ \substack{ U \in \mathcal{U}_o \\ || U - I_p || \geq \epsilon } }
			\hat{\Delta}_K(U) - \hat{\Delta}_K(I_p)
			\geq 0
			\right)
			\\
			\mbox{[from \eqref{supp:eq:unif:conv:hat:Delta}:]} ~
			= &
			P \left(
			o_p(1)
			+
			\sup_{ \substack{ U \in \mathcal{U}_o \\ || U - I_p || \geq \epsilon } }
			D_K(U) - D_K(I_p)
			\geq 0
			\right)
			\\
			\mbox{[from \eqref{supp:eq:unique:maximum:DK}:]} ~
			\leq &
			P \left(
			o_p(1)
			- C_{inf}
			\geq 0
			\right)
			\\
			\underset{K \to \infty}{\overset{}{\to}}
			&
			0.
		\end{align*}
		This concludes the proof of the first equation of the theorem.
		The second equation follows because, from Lemma \ref{supp:lem:basic:bounds}, the matrix $\hat{\bar{\mathrm{cov}}}_{Z,K}^{-1/2}$ converges to $ I_p$.
	\end{proof}

	\section{Asymptotic normality}

	\begin{lem}\label{supp:lem:exp:matrix}
		Assume that conditions  \ref{supp:cond:block:independence}, \ref{supp:cond:max:variance} and \ref{supp:cond:asymptotically:distinct:eigenvalues} hold.
		Then for any sequence $\hat{U}_{Z,K}$ in \eqref{supp:eq:ref:hatUZK}, there exists a sequence $\hat{G}_K \in \mathcal{G}_p$ and a sequence $\hat{V}_{Z,K} \in \mathcal{U}_p$ such that
		\[
		\hat{V}_{Z,K}
		\overset{p}{\underset{K \to \infty}{\to}} 0
		\]
		and, with probability going to $1$ as $K \to \infty$,
		\[
		\exp( S( \hat{V}_{Z,K} ) )
		=
		\hat{G}_K \hat{U}_{Z,K}
		\] 
		and
		\[
		\nabla_{\hat{V}_{Z,K}}  
		= 0,
		\]
		where $\nabla_{V_0} \in \mathcal{U}_p $ is the gradient evaluated at $V_0 \in \mathcal{U}_p$ of the function $ V \to \hat{\Delta}_{K}(\exp( S( V ) ) )$, with $\hat{\Delta}_{K}$ as in the proof of Theorem \ref{supp:thm:consistency}.
		
	\end{lem}
	\begin{proof}
		From Theorem \ref{supp:thm:consistency}, for any sequence $\hat{U}_{Z,K}$ in \eqref{supp:eq:ref:hatUZK}, there exists a sequence $\hat{G}_K \in \mathcal{G}_p$ such that
		\[
		\hat{G}_K \hat{U}_{Z,K}
		\overset{p}{\underset{K \to \infty}{\to}}
		I_p.
		\] 
		Also, from \eqref{supp:eq:ref:hatUZK},
		\[
		\hat{G}_K \hat{U}_{Z,K}
		\in \mathrm{argmax}_{ U \in \mathcal{U}_p }
		\hat{\Delta}_{K} (U).
		\]
		From Chapter 2 in \cite{hall2015lie}, there exists $\epsilon >0$ such that, with $B_{0,\epsilon} = \{ M \in  \mathcal{M}_p ; || M || \leq \epsilon \}$, the exponential function is bijective from $B_{0,\epsilon}$ to $E$, for some set $E$ containing a neighborhood of $I_p$, with reciprocal function the matrix logarithm function $\log$. Hence, any $U \in \mathcal{O}_p \cap E$ can be written as $\exp(S)$ with $S = \log( U )$. We have $\exp(S) \exp(S)' = I_p$ so that, using Proposition 2.3 in \cite{hall2015lie} we obtain 
		\[
		\exp(S') = \exp(S)' = \exp(S)^{-1} = \exp(-S),
		\]
		so that, applying the logarithm, $S' = -S$. Hence, when $\hat{G} \hat{U}_{Z,K} \in E$, we can write $\hat{G} \hat{U}_{Z,K}  = \exp( S( \hat{V}_{Z,K} ) ) $ for $\hat{V}_{Z,K} \in \mathcal{U}_p$. [We can define $\hat{V}_{Z,K}$ arbitrarily on the event where $\hat{G} \hat{U}_{Z,K} \not \in E$ and the probability of this event goes to zero as $K \to \infty$.]
		
		By continuity of the logarithm function around $I_p$ (see Chapter 2 in \cite{hall2015lie}), we thus have $\hat{V}_{Z,K} \to 0$ in probability as $n \to \infty$. Also, we have on the event $\hat{G} \hat{U}_{Z,K} \in E$, since $\exp(S(V)) \in \mathcal{O}_p$ for $V \in \mathcal{U}_p$,
		\[
		\hat{V}_{Z,K} 
		\in \mathrm{argmax}_{ V \in \mathcal{U}_p }
		\hat{\Delta}_K( \exp(S(V)) ),
		\] 
		and so
		\[
		\nabla_{\hat{V}_{Z,K}} 
		= 0. 
		\]
	\end{proof}

	\begin{lem} \label{supp:lem:equi:gradient}
		Assume that conditions  \ref{supp:cond:block:independence}, \ref{supp:cond:max:variance} and \ref{supp:cond:asymptotically:distinct:eigenvalues} hold.
		Let us write $\hat{C}_i = \hat{\mathrm{cov}}_{Z,i} $, $C_i = \mathrm{cov}_{Z,i} $ and $\hat{\bar{C}}_i = \hat{\bar{\mathrm{cov}}}_{Z,K}^{-1/2} \hat{\mathrm{cov}}_{Z,i} \hat{\bar{\mathrm{cov}}}_{Z,K}^{-1/2}$. Let
		\[
		\hat{T} = 
		-
		\frac{1}{2}
		\left( 
		\hat{\bar{\mathrm{cov}}}_{Z,K}
		-
		I_p
		\right).
		\]
		
		Then, for $1 \leq  j < k \leq p$, recalling that $\nabla_0$ is the gradient of the function $ V \to \hat{\Delta}_{K}(\exp( S( V ) ) )$ evaluated at $0$, we have
		\[
		[\nabla_0]_{j,k} 
		=
		[\bar{\nabla}_0]_{j,k} 
		+ O_p(1/K),
		\]
		with
		\begin{align*}
			[\bar{\nabla}_0]_{j,k} 
			= &
			-
			4
			\frac{1}{K}
			\sum_{i=1}^K
			e_k'
			\hat{C}_i
			e_k
			e_k'
			\hat{C}_i
			e_j
			- 8
			e_k'
			\hat{T}
			\left(
			\frac{1}{K}
			\sum_{i=1}^K
			\mathbb{E}
			\left[
			\hat{C}_i
			e_k
			e_k'
			\hat{C}_i
			e_j
			\right]
			\right)
			\\
			& -
			4
			e_k'
			\hat{T}
			\left(
			\frac{1}{K}
			\sum_{i=1}^K
			\mathbb{E}
			\left[
			\left(
			e_k'
			\hat{C}_i
			e_k
			\right)
			\hat{C}_i
			e_j
			\right]
			\right)
			-
			4
			\left(
			\frac{1}{K}
			\sum_{i=1}^K
			\mathbb{E}
			\left[
			e_k'
			\hat{C}_i
			e_k
			e_k'
			\hat{C}_i
			\right]
			\right)
			\hat{T}
			e_j
			\\
			&+
			4
			\frac{1}{K}
			\sum_{i=1}^K
			e_j'
			\hat{C}_i
			e_j
			e_j'
			\hat{C}_i
			e_k
			+ 8
			e_j'
			\hat{T}
			\left(
			\frac{1}{K}
			\sum_{i=1}^K
			\mathbb{E}
			\left[
			\hat{C}_i
			e_j
			e_j'
			\hat{C}_i
			e_k
			\right]
			\right)
			\\
			& +
			4
			e_j'
			\hat{T}
			\left(
			\frac{1}{K}
			\sum_{i=1}^K
			\mathbb{E}
			\left[
			\left(
			e_j'
			\hat{C}_i
			e_j
			\right)
			\hat{C}_i
			e_k
			\right]
			\right)
			+ 4
			\left(
			\frac{1}{K}
			\sum_{i=1}^K
			\mathbb{E}
			\left[
			e_j'
			\hat{C}_i
			e_j
			e_j'
			\hat{C}_i
			\right]
			\right)
			\hat{T}
			e_k. 
		\end{align*}
	\end{lem}
	
	\begin{proof}
		
		In order the compute the gradient of the function $ V \to \hat{\Delta}_{K}(\exp( S( V ) ) )$ at zero we use $\exp(X) = I_p + X  + o(||X||)$ when $X \in \mathcal{M}_p \to 0$. We have, when $V \in \mathcal{U}_p \to 0$, recalling that $M_j'$ denotes the row $j$ of a square matrix $M$,
		\begin{align*}
			\hat{\Delta}_{K}(\exp( S( V ) ) )
			= &
			\frac{1}{K}
			\sum_{i=1}^K
			\sum_{j=1}^p
			\left(
			\exp( S( V ) )_j'
			\hat{\bar{\mathrm{cov}}}_{Z,K}^{-1/2}
			\hat{\mathrm{cov}}_{Z,i}
			\hat{\bar{\mathrm{cov}}}_{Z,K}^{-1/2}
			\exp( S( V ) )_j
			\right)^2
			\\
			= &
			\frac{1}{K}
			\sum_{i=1}^K
			\sum_{j=1}^p
			\left(
			( I_p + S(V) + a(V) )_j'
			\hat{\bar{C}}_i
			( I_p + S(V) + a(V) )_j
			\right)^2
			\\
			= &
			\frac{1}{K}
			\sum_{i=1}^K
			\sum_{j=1}^p
			\left(
			e_j'
			\hat{\bar{C}}_i
			e_j
			+
			2 
			e_j'
			\hat{\bar{C}}_i
			S(V)_j
			+
			b(V)
			\right)^2
			\\
			= &
			\frac{1}{K}
			\sum_{i=1}^K
			\sum_{j=1}^p
			\left(
			e_j'
			\hat{\bar{C}}_i
			e_j
			e_j'
			\hat{\bar{C}}_i
			e_j
			+
			4
			e_j'
			\hat{\bar{C}}_i
			e_j
			e_j'
			\hat{\bar{C}}_i
			S(V)_j
			+
			c(V)
			\right),
		\end{align*}
		where $||a(V)|| = o(||S(V)||)$, $||b(V)|| = o(||S(V)||)$ and $||c(V)|| = o(||S(V)||)$ as $V \to 0$. Hence we have
		\begin{align*}
			\hat{\Delta}_{K}(\exp( S( V ) ) )
			= &
			\frac{1}{K}
			\sum_{i=1}^K
			\sum_{j=1}^p
			e_j'
			\hat{\bar{C}}_i
			e_j
			e_j'
			\hat{\bar{C}}_i
			e_j
			+
			\frac{1}{K}
			\sum_{i=1}^K
			\sum_{j=1}^p
			4
			e_j'
			\hat{\bar{C}}_i
			e_j
			\left(
			\sum_{k=1}^p
			\left[ \hat{\bar{C}}_i \right]_{j,k}
			S(V)_{j,k}
			\right)
			+
			o(||V||)
			\\
			= &
			\frac{1}{K}
			\sum_{i=1}^K
			\sum_{j=1}^p
			e_j'
			\hat{\bar{C}}_i
			e_j
			e_j'
			\hat{\bar{C}}_i
			e_j
			+
			\frac{1}{K}
			\sum_{i=1}^K
			\sum_{j=1}^{p-1}
			\sum_{k=j+1}^{p}
			4
			e_j'
			\hat{\bar{C}}_i
			e_j
			\left[ \hat{\bar{C}}_i \right]_{j,k}
			V_{j,k}
			\\
			& +
			\frac{1}{K}
			\sum_{i=1}^K
			\sum_{k=1}^{p-1}
			\sum_{j=k+1}^{p}
			4
			e_j'
			\hat{\bar{C}}_i
			e_j
			\left[ \hat{\bar{C}}_i \right]_{j,k}
			(-V_{k,j})
			+
			o(||V||)
			\\
			= &
			\frac{1}{K}
			\sum_{i=1}^K
			\sum_{j=1}^p
			e_j'
			\hat{\bar{C}}_i
			e_j 
			e_j'
			\hat{\bar{C}}_i
			e_j
			\\
			& +
			\frac{1}{K}
			\sum_{i=1}^K
			\sum_{j=1}^{p-1}
			\sum_{k=j+1}^{p}
			V_{j,k}
			4
			\left(
			-
			e_k'
			\hat{\bar{C}}_i
			e_k
			e_k'
			\hat{\bar{C}}_i
			e_j
			+
			e_j'
			\hat{\bar{C}}_i
			e_j
			e_j'
			\hat{\bar{C}}_i
			e_k
			\right)
			+
			o(||V||).
		\end{align*}
		Hence, it follows that for $1 \leq  j < k \leq p$,
		\[
		[\nabla_0]_{j,k} =
		4
		\frac{1}{K}
		\sum_{i=1}^K
		\left(
		-
		e_k'
		\hat{\bar{C}}_i
		e_k
		e_k'
		\hat{\bar{C}}_i
		e_j
		+
		e_j'
		\hat{\bar{C}}_i
		e_j
		e_j'
		\hat{\bar{C}}_i
		e_k
		\right).
		\]
		
		Let 
		\[
		T = \hat{\bar{\mathrm{cov}}}_{Z,K}^{-1/2} - I_p.
		\]
		As shown in the proof of Lemma \ref{supp:lem:removing:hat}, and from \eqref{supp:eq:lem:bound:trois}, we have $T = O_p( 1/K^{1/2} )$. We have
		\begin{flalign*}
			& \frac{1}{K}
			\sum_{i=1}^K
			e_k'
			\hat{\bar{C}}_i
			e_k
			e_k'
			\hat{\bar{C}}_i
			e_j
			= &
			\\
			& \frac{1}{K}
			\sum_{i=1}^K
			e_k'
			( I_p + T )
			\hat{C}_i
			( I_p+ T )
			e_k
			e_k'
			( I_p + T )
			\hat{C}_i
			( I_p + T )
			e_j. &
		\end{flalign*}
		In the above display, after expanding the terms $( I_p + T )$, each of the obtained sums containing two times $T$ or more is a $O_p(1/K)$, as can be shown from Lemma \ref{supp:lem:basic:bounds}.  Hence we obtain
		\begin{flalign*}
			& \frac{1}{K}
			\sum_{i=1}^K
			e_k'
			\hat{\bar{C}}_i
			e_k
			e_k'
			\hat{\bar{C}}_i
			e_j
			= &
			\\
			& \frac{1}{K}
			\sum_{i=1}^K
			e_k'
			\hat{C}_i
			e_k
			e_k'
			\hat{C}_i
			e_j
			+
			2
			\frac{1}{K}
			\sum_{i=1}^K
			e_k'
			T
			\hat{C}_i
			e_k
			e_k'
			\hat{C}_i
			e_j
			&
			\\
			& +
			\frac{1}{K}
			\sum_{i=1}^K
			e_k'
			\hat{C}_i
			e_k
			e_k'
			T
			\hat{C}_i
			e_j
			+
			\frac{1}{K}
			\sum_{i=1}^K
			e_k'
			\hat{C}_i
			e_k
			e_k'
			\hat{C}_i
			T
			e_j + O_p(1/K).
		\end{flalign*}
		We finally obtain
		\begin{flalign*}
			&\frac{1}{K}
			\sum_{i=1}^K
			e_k'
			\hat{\bar{C}}_i
			e_k
			e_k'
			\hat{\bar{C}}_i
			e_j
			= &
			\\
			&
			\frac{1}{K}
			\sum_{i=1}^K
			e_k'
			\hat{C}_i
			e_k
			e_k'
			\hat{C}_i
			e_j
			+ 2
			e_k'
			T
			\left(
			\frac{1}{K}
			\sum_{i=1}^K
			\hat{C}_i
			e_k
			e_k'
			\hat{C}_i
			e_j
			\right)
			+
			e_k'
			T
			\left(
			\frac{1}{K}
			\sum_{i=1}^K
			\left(
			e_k'
			\hat{C}_i
			e_k
			\right)
			\hat{C}_i
			e_j
			\right)
			&
			\\
			&
			+
			\left(
			\frac{1}{K}
			\sum_{i=1}^K
			e_k'
			\hat{C}_i
			e_k
			e_k'
			\hat{C}_i
			\right)
			T
			e_j
			+
			O_p(1/K).
			& \\
		\end{flalign*}
		In the three last sums under parenthesis of the above displays, any two of the summands are independent when the corresponding difference of indices is larger or equal to $L$. Furthermore, the norms of these summands have bounded moments from Lemma \ref{supp:lem:basic:bounds}. Also, recall that $T = O_p(K^{-1/2})$. Hence, we obtain
		\begin{flalign*}
			&\frac{1}{K}
			\sum_{i=1}^K
			e_k'
			\hat{\bar{C}}_i
			e_k
			e_k'
			\hat{\bar{C}}_i
			e_j
			= &
			\\
			&
			\frac{1}{K}
			\sum_{i=1}^K
			e_k'
			\hat{C}_i
			e_k
			e_k'
			\hat{C}_i
			e_j
			+ 2
			e_k'
			T
			\left(
			\frac{1}{K}
			\sum_{i=1}^K
			\mathbb{E}
			\left[
			\hat{C}_i
			e_k
			e_k'
			\hat{C}_i
			e_j
			\right]
			\right)
			+
			e_k'
			T
			\left(
			\frac{1}{K}
			\sum_{i=1}^K
			\mathbb{E}
			\left[
			\left(
			e_k'
			\hat{C}_i
			e_k
			\right)
			\hat{C}_i
			e_j
			\right]
			\right)
			&
			\\
			&
			+
			\left(
			\frac{1}{K}
			\sum_{i=1}^K
			\mathbb{E}
			\left[
			e_k'
			\hat{C}_i
			e_k
			e_k'
			\hat{C}_i
			\right]
			\right)
			T
			e_j
			+
			O_p(1/K).
			& \\
		\end{flalign*}
		Hence, we finally have
		\begin{align*}
			[\nabla_0]_{j,k} 
			= &
			- 4
			\frac{1}{K}
			\sum_{i=1}^K
			e_k'
			\hat{C}_i
			e_k
			e_k'
			\hat{C}_i
			e_j
			- 8
			e_k'
			T
			\left(
			\frac{1}{K}
			\sum_{i=1}^K
			\mathbb{E}
			\left[
			\hat{C}_i
			e_k
			e_k'
			\hat{C}_i
			e_j
			\right]
			\right)
			\\
			& -
			4
			e_k'
			T
			\left(
			\frac{1}{K}
			\sum_{i=1}^K
			\mathbb{E}
			\left[
			\left(
			e_k'
			\hat{C}_i
			e_k
			\right)
			\hat{C}_i
			e_j
			\right]
			\right)
			-
			4
			\left(
			\frac{1}{K}
			\sum_{i=1}^K
			\mathbb{E}
			\left[
			e_k'
			\hat{C}_i
			e_k
			e_k'
			\hat{C}_i
			\right]
			\right)
			T
			e_j
			\\
			& +
			4
			\frac{1}{K}
			\sum_{i=1}^K
			e_j'
			\hat{C}_i
			e_j
			e_j'
			\hat{C}_i
			e_k
			+ 8
			e_j'
			T
			\left(
			\frac{1}{K}
			\sum_{i=1}^K
			\mathbb{E}
			\left[
			\hat{C}_i
			e_j
			e_j'
			\hat{C}_i
			e_k
			\right]
			\right)
			\\
			& +
			4
			e_j'
			T
			\left(
			\frac{1}{K}
			\sum_{i=1}^K
			\mathbb{E}
			\left[
			\left(
			e_j'
			\hat{C}_i
			e_j
			\right)
			\hat{C}_i
			e_k
			\right]
			\right)
			+ 4
			\left(
			\frac{1}{K}
			\sum_{i=1}^K
			\mathbb{E}
			\left[
			e_j'
			\hat{C}_i
			e_j
			e_j'
			\hat{C}_i
			\right]
			\right)
			T
			e_k
			+ O_p(1/K). 
		\end{align*}
		From the expression of the derivative of the inverse matrix square root around the identity, we obtain
		\[
		T = \hat{T} + O_p( 1/K).
		\]
		This concludes the proof.
	\end{proof}
	
	Let $d_w$ denote a metric generating the topology of weak convergence on the set of Borel probability measures on Euclidean spaces; for specific examples see, e.g., the discussion in \cite{dudleyreal} p. 393.
	
	\begin{lem} \label{supp:lem:TCL:gradient}
		Assume that conditions  \ref{supp:cond:block:independence}, \ref{supp:cond:max:variance} and \ref{supp:cond:asymptotically:distinct:eigenvalues} hold.
		Let $\Sigma_{\nabla}$ be the covariance matrix of $ K^{1/2} \bar{\nabla}_0$. Let $Q_K$ be the distribution of $K^{1/2} \bar{\nabla}_0$. Then, as $K \to \infty$ we have $d_w( Q_K , \mathcal{N}( 0 , \Sigma_{\nabla} ) ) \to 0$ as $K \to \infty$. Furthermore, the matrix $\Sigma_{\nabla}$ is bounded as $K \to \infty$.
		
	\end{lem}
	
	\begin{proof}
		
		We can write, for $1 \leq j < k \leq p$,
		\begin{equation*} 
			[ \bar{\nabla}_0]_{j,k}
			=
			\frac{1}{K}
			\sum_{i=1}^K V_{i}^{(j,k)}
			+
			e_k' \hat{T}
			\frac{1}{K}
			\sum_{i=1}^K W_{i}^{(j,k)}
			+
			e_j' \hat{T}
			\frac{1}{K}
			\sum_{i=1}^K X_{i}^{(j,k)},
		\end{equation*}
		where $V_{i}^{(j,k)} \in \mathbb{R} $ is random and $W_{i}^{(j,k)}$ and $X_{i}^{(j,k)}$ are fixed $p \times 1$ vectors.
		Furthermore, since, with the notation of Lemma \ref{supp:lem:equi:gradient},
		\[
		\hat{\bar{\mathrm{cov}}}_{Z,K}
		-
		I_p
		=
		\frac{1}{K}
		\sum_{i=1}^K
		\left(
		\hat{C}_i
		-
		\mathbb{E}[ \hat{C}_i ]
		\right),
		\]
		we have
		\begin{equation} \label{supp:eq:for:CLT:dependent}
			[ \bar{\nabla}_0]_{j,k}
			=
			\frac{1}{K}
			\sum_{i=1}^K 
			\left(
			V_{i}^{(j,k)}
			- \frac{1}{2} 
			e_k'
			\left(
			\hat{C}_i
			-
			\mathbb{E}[ \hat{C}_i ]
			\right)
			\left[
			\frac{1}{K}
			\sum_{a=1}^K W_{a}^{(j,k)}
			\right]
			- \frac{1}{2} 
			e_j'
			\left(
			\hat{C}_i
			-
			\mathbb{E}[ \hat{C}_i ]
			\right)
			\left[
			\frac{1}{K}
			\sum_{a=1}^K X_{a}^{(j,k)}
			\right]
			\right).
		\end{equation}
		The vectors $\{ ( V_{i}^{(j,k)} )_{1 \leq j < k \leq p} \}_{i \in \mathbb{N}}$ 
		are such that $((V_{i}^{(j,k)})_{j,k} , \hat{C}_i )$ and $((V_{i'}^{(j,k)})_{j,k} , \hat{C}_{i'})$ are independent for $|i - i'| \geq L$. Furthermore, from Lemma \ref{supp:lem:basic:bounds},
		\[
		\sup_{i \in \mathbb{N}, 1 \leq j < k \leq p}
		\mathbb{E} ( | V_{i}^{(j,k)}|^r ) \leq C_{sup}
		~ ~ \mbox{for any fixed $r >0$ },
		\]
		\[
		\sup_{i \in \mathbb{N}, 1 \leq j < k \leq p}
		||   W_{i}^{(j,k)} ||  \leq C_{sup}
		\] 
		and
		\[
		\sup_{i \in \mathbb{N}, 1 \leq j < k \leq p}
		||   X_{i}^{(j,k)} ||  \leq C_{sup}.
		\]
		Thus, the quantity in \eqref{supp:eq:for:CLT:dependent} is a component of an average of random vectors, with bounded moments, and such that two of these vectors are independent if their index difference is larger or equal to $L$.
		Hence the matrix $\Sigma_{\nabla}$ is bounded. Thus, one can first assume that the sequence of matrices $\Sigma_{\nabla}$ converges as $K \to \infty$, up to taking a subsequence. Then, one can apply a central limit theorem for weakly dependent variables (for instance Theorem 2.1 in \cite{Neumann13}) to \eqref{supp:eq:for:CLT:dependent}.
		
		This proves that, with $Q'_K$ be the distribution of 
		\[
		\sqrt{K}
		\left(
		\left[
		\frac{1}{K}
		\sum_{i=1}^K \left( V_{i}^{(j,k)} - \mathbb{E}[ V_{i}^{(j,k)} ] \right) 
		+
		e_k' \hat{T}
		\frac{1}{K}
		\sum_{i=1}^K W_{i}^{(j,k)}
		+
		e_j' \hat{T}
		\frac{1}{K}
		\sum_{i=1}^K X_{i}^{(j,k)}
		\right]_{1 \leq j < k \leq p}
		\right),
		\]
		we have $d_w( Q'_K , \mathcal{N}( 0 , \Sigma_{\nabla} ) ) \to 0$ as $K \to \infty$. Hence, we can conclude the proof by showing that $\mathbb{E}[ V_{i}^{(j,k)} ]= 0$ for $1 \leq j < k \leq p$. 
		We have, for $1 \leq j < k \leq p$,
		\begin{align*}
			\mathbb{E}[
			e_k'
			\hat{C}_i
			e_k
			e_k'
			\hat{C}_i
			e_j
			]
			= &
			\sum_{a,b,c,d=1}^p
			(e_k)_a
			(e_k)_b
			(e_k)_c
			(e_j)_d
			\mathbb{E}[
			(\hat{C}_i)_{a,b}
			(\hat{C}_i)_{c,d}
			].
		\end{align*}
		Since $(e_k)_l$ is zero for $l \neq k$, we obtain from \eqref{supp:eq:moment:product:hat:cov}
		\begin{align*}
			\mathbb{E}[
			e_k'
			\hat{C}_i
			e_k
			e_k'
			\hat{C}_i
			e_j
			]
			= &
			\mathbb{E}[
			(\hat{C}_i)_{k,k}
			(\hat{C}_i)_{k,j}
			]
			\\
			= &
			\left( \mathrm{cov}_{Z,i} \right)_{k,k}
			\left( \mathrm{cov}_{Z,i} \right)_{k,j}
			\\
			& +
			\mathbf{1}_{k=k}
			\mathbf{1}_{k=j}
			\frac{1}{s^2}
			\sum_{a,b = 1}^s
			(D^{(a,b)}_{Z,i})_{k,k}
			(D^{(a,b)}_{Z,i})_{k,k}
			\notag
			\\
			& +
			\mathbf{1}_{k=j}
			\mathbf{1}_{k=k}
			\frac{1}{s^2}
			\sum_{a,b = 1}^s
			(D^{(a,b)}_{Z,i})_{k,k}
			\mathbb{E}
			(D^{(a,b)}_{Z,i})_{k,k}
			\\
			= &
			0, \notag
		\end{align*}
		since $1 \leq j < k \leq p$. Similarly, we show
		\[ 
		\mathbb{E}[
		e_j'
		\hat{C}_i
		e_j
		e_j'
		\hat{C}_i
		e_k
		] = 0
		\]
		for $1 \leq j < k \leq p$. Hence the proof is concluded. 
		
	\end{proof}

	\begin{lem} \label{supp:lem:Hessian:equi}
		Assume that conditions  \ref{supp:cond:block:independence}, \ref{supp:cond:max:variance} and \ref{supp:cond:asymptotically:distinct:eigenvalues} hold.
		Let us write, for $1 \leq e < f \leq p$ and $1 \leq g < h \leq p$, 
		\[
		E_{e,f,g,h} =
		\frac{ \partial^2 }{ \partial V_{e,f} \partial V_{g,h}
		}
		\left. \hat{\Delta}_K( \exp(S(V)) ) \right|_{V=0},
		\]
		that is, $E_{e,f,g,h}$ is the element $(e,f) \times (g,h)$ of the Hessian matrix of the function $ V \to \hat{\Delta}_{K}(\exp( S( V ) ) )$ at $V=0$. 
		Then, for $1 \leq e < f \leq p$ and $1 \leq g < h \leq p$, we have
		\begin{align*}
			E_{e,f,g,h}
			= &
			o_p(1) -  4
			\mathbf{1}_{(e,f) = (g,h)}
			\frac{1}{K}
			\sum_{i=1}^K
			\left(
			[C_i]_{e,e}
			-
			[C_i]_{f,f}
			\right)^2
			\\
			& 
			- 
			8
			\mathbf{1}_{(e,f) = (g,h)}
			\frac{1}{K}
			\sum_{i=1}^K
			\frac{1}{s^2} 
			\sum_{m,n=1}^s
			\left(
			\left[
			D_{Z,i}^{(m,n)}
			\right]_{e,e}
			-
			\left[
			D_{Z,i}^{(m,n)}
			\right]_{f,f}
			\right)^2.
		\end{align*}
		
	\end{lem}
	
	\begin{proof}
		
		In order to calculate the Hessian matrix at $V=0$, we use a Taylor expansion as $V \in \mathcal{U}_p \to 0$. We have, with $||a(V)|| = O( ||V||^3 )$ and $||b(V)|| = O( ||V||^3 )$ as $V \to 0$,
		
		\begin{align*}
			\hat{\Delta}_K( \exp(S(V)) )
			= &
			\frac{1}{K}
			\sum_{i=1}^K
			\sum_{j=1}^p
			\left(
			\left[
			I_p + S(V) + \frac{1}{2} S(V)^2 + a(V)  
			\right]_j'
			\hat{\bar{C}}_i
			\left[ I_p + S(V) + \frac{1}{2} S(V)^2 + a(V)  
			\right]_j
			\right)^2
			\\
			= &
			\frac{1}{K}
			\sum_{i=1}^K
			\sum_{j=1}^p
			\left(
			e_j'
			\hat{\bar{C}}_i
			e_j
			+ 2 S(V)_j' \hat{\bar{C}}_i e_j
			+ 2 \frac{1}{2} [S(V)^2]'_j \hat{\bar{C}}_i
			e_j
			+ 
			[S(V)]'_j \hat{\bar{C}}_i [S(V)]_j
			+ b(V)
			\right)^2.
		\end{align*}
		In the following, we gather the second order terms:
		\begin{flalign*}
			&
			\frac{1}{K}
			\sum_{i=1}^K
			\sum_{j=1}^p
			\left(
			2 e_j' \hat{\bar{C}}_i e_j
			[S(V)^2]'_j \hat{\bar{C}}_i e_j
			+ 2 e_j' \hat{\bar{C}}_i e_j
			[S(V)]'_j \hat{\bar{C}}_i [S(V)]_j
			+ 4 
			[S(V)]'_j \hat{\bar{C}}_i e_j
			e_j' \hat{\bar{C}}_i  [S(V)]_j 
			\right)
			&
			\\
			&
			=
			\frac{1}{K}
			\sum_{i=1}^K
			\sum_{j=1}^p
			\left(
			2 e_j' \hat{\bar{C}}_i e_j
			\left(
			\sum_{k,a = 1}^p
			S(V)_{j,k} S(V)_{k,a} [ \hat{\bar{C}}_i]_{a,j}
			\right)
			+ 2 e_j' \hat{\bar{C}}_i e_j
			\left(
			\sum_{k,a = 1}^p
			S(V)_{j,k} S(V)_{j,a} [ \hat{\bar{C}}_i]_{k,a}
			\right)
			\right.
			&
			\\
			&
			\left.
			~ ~
			+
			4
			\left(
			\sum_{a,b = 1}^p
			[ \hat{\bar{C}}_i]_{j,a}
			S(V)_{j,a} 
			S(V)_{j,b} 
			[ \hat{\bar{C}}_i]_{b,j}
			\right)
			\right)
			&
			\\
			&
			=
			\frac{1}{K}
			\sum_{i=1}^K
			\sum_{j=1}^p
			\sum_{k,a = 1}^p
			\left(
			- 2 [\hat{\bar{C}}_i]_{j,j}
			[\hat{\bar{C}}_i]_{a,j}
			S(V)_{a,k} S(V)_{j,k}
			+ 2 [\hat{\bar{C}}_i]_{k,k}
			[\hat{\bar{C}}_i]_{a,j}
			S(V)_{a,k} S(V)_{j,k}
			\right.
			&
			\\
			&
			\left.
			~ ~ + 4
			[\hat{\bar{C}}_i]_{k,a}
			[\hat{\bar{C}}_i]_{j,k}
			S(V)_{a,k} S(V)_{j,k}
			\right)
			&
			\\
			&
			=
			\frac{1}{K}
			\sum_{i=1}^K
			\sum_{j=1}^p
			\sum_{k,a = 1}^p
			\left( 
			- 2 [\hat{\bar{C}}_i]_{j,j}
			[\hat{\bar{C}}_i]_{a,j}
			+ 2 [\hat{\bar{C}}_i]_{k,k}
			[\hat{\bar{C}}_i]_{a,j}
			+ 4
			[\hat{\bar{C}}_i]_{k,a}
			[\hat{\bar{C}}_i]_{j,k}
			\right)
			S(V)_{a,k} S(V)_{j,k}.
			&
		\end{flalign*}
		From the above display, it follows that, for $1 \leq e < f \leq p$ and $1 \leq g < h \leq p$, 
		\begin{flalign*}
			&
			E_{e,f,g,h}
			= 
			&
			\\
			&
			\frac{1}{K}
			\sum_{i=1}^K
			\sum_{j=1}^p
			\sum_{k,a = 1}^p
			\left( 
			- 4  [\hat{\bar{C}}_i]_{j,j}
			[\hat{\bar{C}}_i]_{a,j}
			+ 4  [\hat{\bar{C}}_i]_{k,k}
			[\hat{\bar{C}}_i]_{a,j}
			+ 8
			[\hat{\bar{C}}_i]_{k,a}
			[\hat{\bar{C}}_i]_{j,k}
			\right)
			&
			\\
			&
			\left(
			\mathbf{1}_{(a,k) = (e,f)}
			-
			\mathbf{1}_{(a,k) = (f,e)}
			\right)
			\left(
			\mathbf{1}_{(j,k) = (g,h)}
			-
			\mathbf{1}_{(j,k) = (h,g)}
			\right).
			&
		\end{flalign*}

		As in the proof of Lemma \ref{supp:lem:equi:gradient}, we can show that we have
		\begin{flalign*}
			&
			E_{e,f,g,h}
			= 
			o_p(1) + 
			&
			\\
			&
			\frac{1}{K}
			\sum_{i=1}^K
			\sum_{j=1}^p
			\sum_{k,a = 1}^p
			\left( 
			- 
			4 
			[\hat{C}_i]_{j,j}
			[\hat{C}_i]_{a,j}
			+ 
			4
			[\hat{C}_i]_{k,k}
			[\hat{C}_i]_{a,j}
			+ 
			8
			[\hat{C}_i]_{k,a}
			[\hat{C}_i]_{j,k}
			\right)
			&
			\\
			&
			\left(
			\mathbf{1}_{(a,k) = (e,f)}
			-
			\mathbf{1}_{(a,k) = (f,e)}
			\right)
			\left(
			\mathbf{1}_{(j,k) = (g,h)}
			-
			\mathbf{1}_{(j,k) = (h,g)}
			\right).
		\end{flalign*}
		Then, since $\hat{C}_i$ and $\hat{C}_j$ are independent for $|i - j| \geq L$ and from \eqref{supp:eq:lem:bound:deux}, we obtain
		\begin{flalign*}
			&
			E_{e,f,g,h}
			= 
			o_p(1) + 
			&
			\\
			&
			\frac{1}{K}
			\sum_{i=1}^K
			\sum_{j=1}^p
			\sum_{k,a = 1}^p
			\left( 
			- 
			4
			\mathbb{E} \left[
			[\hat{C}_i]_{j,j}
			[\hat{C}_i]_{a,j}
			\right]
			+ 
			4
			\mathbb{E} \left[
			[\hat{C}_i]_{k,k}
			[\hat{C}_i]_{a,j}
			\right]
			+ 
			8
			\mathbb{E} \left[
			[\hat{C}_i]_{k,a}
			[\hat{C}_i]_{j,k}
			\right]
			\right)
			&
			\\
			&
			\left(
			\mathbf{1}_{(a,k) = (e,f)}
			-
			\mathbf{1}_{(a,k) = (f,e)}
			\right)
			\left(
			\mathbf{1}_{(j,k) = (g,h)}
			-
			\mathbf{1}_{(j,k) = (h,g)}
			\right)
			&
			\\
			& = o_p(1) +  S_1 + S_2 + S_3,
		\end{flalign*}
		say. We have, from \eqref{supp:eq:moment:product:hat:cov}, and recalling that $1 \leq e < f \leq p$ and $1 \leq g < h \leq p$,
		\begin{align*}
			S_1  
			= &
			\frac{1}{K}
			\sum_{i=1}^K
			\sum_{j=1}^p
			\sum_{k,a = 1}^p
			\left( 
			- 
			4
			\mathbb{E} \left[
			[\hat{C}_i]_{j,j}
			[\hat{C}_i]_{a,j}
			\right]
			\right)
			\left(
			\mathbf{1}_{(a,k) = (e,f)}
			-
			\mathbf{1}_{(a,k) = (f,e)}
			\right)
			\left(
			\mathbf{1}_{(j,k) = (g,h)}
			-
			\mathbf{1}_{(j,k) = (h,g)}
			\right)
			\\
			= &
			\frac{1}{K}
			\sum_{i=1}^K
			\sum_{j=1}^p
			\sum_{k,a = 1}^p
			\left( 
			-  4
			\right)
			\left(
			\mathbf{1}_{a=j}
			[C_i]_{j,j}^2
			+
			2
			\mathbf{1}_{a=j}
			\frac{1}{s^2} 
			\sum_{m,n=1}^s
			\left[
			D_{Z,i}^{(m,n)}
			\right]^2_{j,j}
			\right)
			\\
			&
			\left(
			\mathbf{1}_{(a,k) = (e,f)}
			-
			\mathbf{1}_{(a,k) = (f,e)}
			\right)
			\left(
			\mathbf{1}_{(j,k) = (g,h)}
			-
			\mathbf{1}_{(j,k) = (h,g)}
			\right)
			\\
			= &
			\mathbf{1}_{(e,f) = (g,h)}
			\frac{1}{K}
			\sum_{i=1}^K
			\left(
			- 
			4
			[C_i]_{e,e}^2
			-
			8
			\frac{1}{s^2} 
			\sum_{m,n=1}^s
			\left[
			D_{Z,i}^{(m,n)}
			\right]^2_{e,e}
			- 
			4
			[C_i]_{f,f}^2
			-
			8
			\frac{1}{s^2} 
			\sum_{m,n=1}^s
			\left[
			D_{Z,i}^{(m,n)}
			\right]^2_{f,f}
			\right).
		\end{align*}
		Then, we have
		\begin{align*}
			S_2  
			= &
			\frac{1}{K}
			\sum_{i=1}^K
			\sum_{j=1}^p
			\sum_{k,a = 1}^p
			\left( 
			4
			\mathbb{E} \left[
			[\hat{C}_i]_{k,k}
			[\hat{C}_i]_{a,j}
			\right]
			\right)
			\left(
			\mathbf{1}_{(a,k) = (e,f)}
			-
			\mathbf{1}_{(a,k) = (f,e)}
			\right)
			\left(
			\mathbf{1}_{(j,k) = (g,h)}
			-
			\mathbf{1}_{(j,k) = (h,g)}
			\right)
			\\
			= &
			\frac{1}{K}
			\sum_{i=1}^K
			\sum_{j=1}^p
			\sum_{k,a = 1}^p
			4
			\left(
			\mathbf{1}_{a=j}
			[C_i]_{k,k}
			[C_i]_{a,j}
			+2
			\mathbf{1}_{k=a}
			\mathbf{1}_{k=j}
			\frac{1}{s^2} 
			\sum_{m,n=1}^s
			\left[
			D_{Z,i}^{(m,n)}
			\right]^2_{k,k}
			\right)
			\\
			&
			\left(
			\mathbf{1}_{(a,k) = (e,f)}
			-
			\mathbf{1}_{(a,k) = (f,e)}
			\right)
			\left(
			\mathbf{1}_{(j,k) = (g,h)}
			-
			\mathbf{1}_{(j,k) = (h,g)}
			\right)
			\\
			= &
			\frac{1}{K}
			\sum_{i=1}^K
			\sum_{j=1}^p
			\sum_{k = 1}^p
			4
			\left(
			[C_i]_{k,k}
			[C_i]_{j,j}
			\right)
			\left(
			\mathbf{1}_{(j,k) = (e,f)}
			-
			\mathbf{1}_{(j,k) = (f,e)}
			\right)
			\left(
			\mathbf{1}_{(j,k) = (g,h)}
			-
			\mathbf{1}_{(j,k) = (h,g)}
			\right)
			\\
			= &
			\mathbf{1}_{(e,f) = (g,h)}
			8
			\frac{1}{K}
			\sum_{i=1}^K
			[C_i]_{e,e}
			[C_i]_{f,f}.
		\end{align*}
		
		We have
		\begin{align*}
			S_3  
			= &
			\frac{1}{K}
			\sum_{i=1}^K
			\sum_{j=1}^p
			\sum_{k,a = 1}^p
			\left( 
			8
			\mathbb{E} \left[
			[\hat{C}_i]_{k,a}
			[\hat{C}_i]_{j,k}
			\right]
			\right)
			\left(
			\mathbf{1}_{(a,k) = (e,f)}
			-
			\mathbf{1}_{(a,k) = (f,e)}
			\right)
			\left(
			\mathbf{1}_{(j,k) = (g,h)}
			-
			\mathbf{1}_{(j,k) = (h,g)}
			\right)
			\\
			= &
			\frac{1}{K}
			\sum_{i=1}^K
			\sum_{j=1}^p
			\sum_{k,a = 1}^p
			8
			\left(
			[C_i]_{k,a}
			[C_i]_{j,k}
			+
			\mathbf{1}_{k=j}
			\mathbf{1}_{a=k}
			\frac{1}{s^2} 
			\sum_{m,n=1}^s
			\left[
			D_{Z,i}^{(m,n)}
			\right]_{k,k}
			\left[
			D_{Z,i}^{(m,n)}
			\right]_{j,j}
			\right.
			\\
			&
			\left.
			+
			\mathbf{1}_{k=k}
			\mathbf{1}_{a=j}
			\frac{1}{s^2} 
			\sum_{m,n=1}^s
			\left[
			D_{Z,i}^{(m,n)}
			\right]_{k,k}
			\left[
			D_{Z,i}^{(m,n)}
			\right]_{a,a}
			\right)
			\\
			&
			\left(
			\mathbf{1}_{(a,k) = (e,f)}
			-
			\mathbf{1}_{(a,k) = (f,e)}
			\right)
			\left(
			\mathbf{1}_{(j,k) = (g,h)}
			-
			\mathbf{1}_{(j,k) = (h,g)}
			\right)
			\\
			= &
			\frac{1}{K}
			\sum_{i=1}^K
			\sum_{j=1}^p
			\sum_{k = 1}^p
			8
			\left(
			\frac{1}{s^2} 
			\sum_{m,n=1}^s
			\left[
			D_{Z,i}^{(m,n)}
			\right]_{k,k}
			\left[
			D_{Z,i}^{(m,n)}
			\right]_{j,j}
			\right)
			\\
			&
			\left(
			\mathbf{1}_{(j,k) = (e,f)}
			-
			\mathbf{1}_{(j,k) = (f,e)}
			\right)
			\left(
			\mathbf{1}_{(j,k) = (g,h)}
			-
			\mathbf{1}_{(j,k) = (h,g)}
			\right)
			\\
			= &
			\mathbf{1}_{(e,f) = (g,h)}
			16
			\frac{1}{K}
			\sum_{i=1}^K
			\frac{1}{s^2} 
			\sum_{m,n=1}^s
			\left[
			D_{Z,i}^{(m,n)}
			\right]_{e,e}
			\left[
			D_{Z,i}^{(m,n)}
			\right]_{f,f}.
		\end{align*}
		
		Putting together the expressions of $S_1$, $S_2$ and $S_3$, we obtain, for $1 \leq e < f \leq p$ and $1 \leq g < h \leq p$,
		\begin{align*}
			E_{e,f,g,h}
			= &
			o_p(1) - 
			4
			\mathbf{1}_{(e,f) = (g,h)}
			\frac{1}{K}
			\sum_{i=1}^K
			\left(
			[C_i]_{e,e}
			-
			[C_i]_{f,f}
			\right)^2
			\\
			& 
			- 
			8
			\mathbf{1}_{(e,f) = (g,h)}
			\frac{1}{K}
			\sum_{i=1}^K
			\frac{1}{s^2} 
			\sum_{m,n=1}^s
			\left(
			\left[
			D_{Z,i}^{(m,n)}
			\right]_{e,e}
			-
			\left[
			D_{Z,i}^{(m,n)}
			\right]_{f,f}
			\right)^2.
		\end{align*}
	\end{proof}
	
	\begin{lem} \label{supp:lem:third:derivative:bounded}
		Assume that conditions  \ref{supp:cond:block:independence}, \ref{supp:cond:max:variance} and \ref{supp:cond:asymptotically:distinct:eigenvalues} hold.
		Let us write, for $1 \leq e < f \leq p$, $1 \leq g < h \leq p$ and $1 \leq k < l \leq p$ and for $V \in \mathcal{U}_p$,
		\[
		F_{e,f,g,h,k,l}(V) =
		\frac{ \partial^3 }{ \partial V_{e,f} \partial V_{g,h}
			\partial V_{k,l}
		}
		\left. \hat{\Delta}_K( \exp(S(V)) ) \right|_{V},
		\]
		that is, $F_{e,f,g,h,k,l}(V)$ is one of the third order partial derivatives of the function $ V \to \hat{\Delta}_{K}(\exp( S( V ) ) )$, evaluated at $V \in \mathcal{U}_p$. Then, there exists $\epsilon >0$ such that
		\[
		\max_{ \substack{
				1 \leq e < f \leq p \\ 
				1 \leq g < h \leq p \\
				1 \leq k < l \leq p
		} }
		\sup_{||S(V)|| \leq \epsilon}
		\left| F_{e,f,g,h,k,l}(V) \right|
		=
		O_p(1).
		\]

	\end{lem}
	
	\begin{proof}
		
		We have, with the notation of Lemma \ref{supp:lem:equi:gradient},
		\begin{align*}
			\hat{\Delta}_K(\exp(S(V)))
			= &
			\frac{1}{K}
			\sum_{i=1}^K
			\sum_{j=1}^p
			\left(
			\exp(S(V))_j'
			\hat{\bar{\mathrm{cov}}}_{Z,K}^{-1/2}
			\hat{\mathrm{cov}}_{Z,i}
			\hat{\bar{\mathrm{cov}}}_{Z,K}^{-1/2}
			\exp(S(V))_j
			\right)^2
			\\
			= &
			\frac{1}{K}
			\sum_{i=1}^K
			\sum_{j=1}^p
			\sum_{a,b,c,d=1}^p
			\exp(S(V))_{j,a}
			\left[
			\hat{\bar{C}}_i
			\right]_{a,b}
			\exp(S(V))_{j,b}
			\exp(S(V))_{j,c}
			\left[
			\hat{\bar{C}}_i
			\right]_{c,d}
			\exp(S(V))_{j,d}.
		\end{align*}
		Hence, it is sufficient to show that, for any $j,a,b,c,d  =1,...,p$ and any $1 \leq e < f \leq p$, $1 \leq g < h \leq p$ and $1 \leq k < l \leq p$,
		\begin{flalign*}
			&
			\sup_{||S(V)|| \leq \epsilon}
			\left|
			\frac{1}{K}
			\sum_{i=1}^K
			\frac{ \partial^3 }{ \partial V_{e,f} \partial V_{g,h}
				\partial V_{k,l}
			}
			\left.
			\left(
			\exp(S(V))_{j,a}
			\left[
			\hat{\bar{C}}_i
			\right]_{a,b}
			\exp(S(V))_{j,b}
			\exp(S(V))_{j,c}
			\left[
			\hat{\bar{C}}_i
			\right]_{c,d}
			\exp(S(V))_{j,d}
			\right)
			\right|_{V}
			\right|
			&
			\\
			&
			= O_p(1).
			&
		\end{flalign*}
		We let 
		\[
		m(V) = \exp(S(V))_{j,a} \exp(S(V))_{j,b} \exp(S(V))_{j,c} \exp(S(V))_{j,d}.
		\]
		From e.g. Proposition 3.2.1 in \cite{hilgert2011structure}, the function $m$ is infinitely differentiable and so it is sufficient to show
		\[
		\sup_{||S(V)|| \leq \epsilon}
		\left|
		\frac{ \partial^3 }{ \partial V_{e,f} \partial V_{g,h}
			\partial V_{k,l}
		}
		m(V)|_V 
		\right|
		\left|
		\frac{1}{K}
		\sum_{i=1}^K
		\left[
		\hat{\bar{C}}_i
		\right]_{a,b}
		\left[
		\hat{\bar{C}}_i
		\right]_{c,d}
		\right|
		= O_p(1).
		\]
		The last display holds true because of Lemma \ref{supp:lem:basic:bounds}.
	\end{proof}
	
	\begin{lem} \label{supp:lem:non:zero:hessian}
		Assume that Condition \ref{supp:cond:asymptotically:distinct:eigenvalues} holds.
		With the notation of Lemma \ref{supp:lem:Hessian:equi}, let, for any $e,f=1,...,p$, $e \neq f$, 
		\begin{align*}
			H_{e,f}
			=
			4
			\frac{1}{K}
			\sum_{i=1}^K
			\left(
			[C_i]_{e,e}
			-
			[C_i]_{f,f}
			\right)^2
			+
			8
			\frac{1}{K}
			\sum_{i=1}^K
			\frac{1}{s^2} 
			\sum_{m,n=1}^s
			\left(
			\left[
			D_{Z,i}^{(m,n)}
			\right]_{e,e}
			-
			\left[
			D_{Z,i}^{(m,n)}
			\right]_{f,f}
			\right)^2.
		\end{align*}
		Then we have, for any $e,f=1,...,p$, $e \neq f$,
		\[
		\liminf_{K \to \infty}
		H_{e,f}
		>0.
		\]
	\end{lem}
	\begin{proof}
		The lemma is a direct consequence of Condition \ref{supp:cond:asymptotically:distinct:eigenvalues}.
	\end{proof}
	
	\begin{thm} \label{supp:thm:TCL}
		
		Assume that conditions  \ref{supp:cond:block:independence}, \ref{supp:cond:max:variance} and \ref{supp:cond:asymptotically:distinct:eigenvalues} hold.
		Let $\Sigma_{\nabla}$ be as in Lemma \ref{supp:lem:TCL:gradient}. Let $\Sigma_{\hat{U}}$ be the $p^2 \times p^2$ covariance matrix defined by, for $ e,f,g,h =1,...,p$,
		\begin{equation*}
			[\Sigma_{\hat{U}}]_{(e,f),(g,h)}
			=
			\begin{cases}
				0
				&
				~ ~ \mbox{if} ~ ~
				e=f ~ \mbox{or} ~ g=h
				\\
				\frac{1}{H_{e,f} H_{g,h}}
				[\Sigma_{\nabla}]_{(e,f),(g,h)}
				&
				~ ~ \mbox{if} ~ ~
				e < f \; , \; g < h
				\\
				\frac{1}{H_{e,f} H_{g,h}}
				\left( - [\Sigma_{\nabla}]_{(e,f),(h,g)} \right)
				&
				~ ~ \mbox{if} ~ ~
				e < f \; , \; g > h
				\\
				\frac{1}{H_{e,f} H_{g,h}}
				\left( - [\Sigma_{\nabla}]_{(f,e),(g,h)} \right)
				&
				~ ~ \mbox{if} ~ ~
				e > f \; , \; g < h
				\\
				\frac{1}{H_{e,f} H_{g,h}}
				[\Sigma_{\nabla}]_{(f,e),(h,g)}
				&
				~ ~ \mbox{if} ~ ~
				e > f \; , \; g > h.
			\end{cases}
		\end{equation*}
		Then for any sequence $\hat{U}_{Z,K}$ in \eqref{supp:eq:ref:hatUZK}, there exists a sequence $\hat{G}_K \in \mathcal{G}_p$ such that, with $Q_{\hat{U}}$ the distribution of $\sqrt{K} ( \hat{G}_K \hat{U}_{Z,K} - I_p )$, we have
		\[
		d_w 
		\left( 
		Q_{\hat{U}} , \mathcal{N}(0,\Sigma_{\hat{U}}) 
		\right)
		\underset{K \to \infty}{\overset{}{\to}} 0.
		\]
		Furthermore, the matrix $\Sigma_{\hat{U}}$ is bounded as $K \to \infty$. 
	\end{thm}
	
	\begin{proof}
		Because of Lemma \ref{supp:lem:non:zero:hessian} and of the fact that the matrix $\Sigma_{\nabla}$ is bounded, the matrix $\Sigma_{\hat{U}}$ is bounded. 
		With the notation of Lemma \ref{supp:lem:exp:matrix}, we have
		\[
		\nabla_{ \hat{V}_{Z,K} } = 0,
		\]
		with probability going to one as $K \to \infty$, where  $\hat{V}_{Z,K} \to 0$ in probability as $K \to \infty$. 
		We consider the event $\nabla_{ \hat{V}_{Z,K} } = 0$ in the rest of the proof.
		We will use a Taylor expansion argument that is classical in M-estimation, but that is here somehow technical to write because we manipulate matrices.
		There exist $p(p-1)/2$ elements of $\mathcal{U}_p$ of the form
		\[
		\{ \tilde{V}_{Z,K,i,j} , 1 \leq i < j <p \},
		\]
		such that each of these elements belongs to the segment with endpoints $0$ and $\hat{V}_{Z,K}$, and such that
		\begin{equation} \label{supp:eq:taylor:hat:v}
			0 = \nabla_0 + \mathcal{E} \left[   \hat{V}_{Z,K}  \right]
			+ \frac{1}{2} \mathcal{F}[ \hat{V}_{Z,K}  ],
		\end{equation}
		where $\mathcal{E}$ is the linear application on $\mathcal{U}_p$ defined by
		\[
		[ \mathcal{E}(V) ]_{a,b}
		=
		\sum_{ \substack{ e,f = 1, \ldots ,p \\ e<f }}
		E_{a,b,e,f}
		V_{e,f},
		\]
		for $1 \leq a < b \leq p$
		with $E$ as in Lemma \ref{supp:lem:Hessian:equi} and where $ \mathcal{F}$ is the quadratic application from 
		$\mathcal{U}_p$ to $\mathcal{U}_p$ such that, for $1 \leq i < j \leq p$, we have
		\[
		(\mathcal{F}[V] )_{i,j}
		=
		\sum_{ \substack{ e,f = 1, \ldots ,p \\ e<f }}
		\sum_{ \substack{ e',f' = 1, \ldots ,p \\ e'<f' }}
		F_{e,f,e',f',i,j}( \tilde{V}_{Z,K,i,j} )
		V_{e,f}
		V_{e',f'},
		\]
		with $F$ as in Lemma \ref{supp:lem:third:derivative:bounded}.
		Remark that, from Lemma \ref{supp:lem:Hessian:equi} and with the notation of Lemma \ref{supp:lem:non:zero:hessian}, we have 
		$\mathcal{E} = - \bar{\mathcal{E}} + o_p(1)$ where $\bar{\mathcal{E}}$ is the linear application on $\mathcal{U}_0$ defined by,
		for $1 \leq e < f \leq p$ 
		\[
		(\bar{\mathcal{E}}[V])_{e,f}
		=
		H_{e,f} V_{e,f}.
		\] 
		Furthermore, from  Lemma \ref{supp:lem:non:zero:hessian}, $\bar{\mathcal{E}}$ is invertible for $K$ large enough and we have from \eqref{supp:eq:taylor:hat:v} that
		\[
		\left[ \bar{\mathcal{E}}^{-1} + \mathcal{R}_1 \right]
		\left[ \nabla_0 \right] = \hat{V}_{Z,K}
		-
		\frac{1}{2}
		\left[ \bar{\mathcal{E}}^{-1} + \mathcal{R}_1 \right]
		\left[
		\mathcal{F}[ \hat{V}_{Z,K}  ]
		\right],
		\]
		where $\mathcal{R}_1 = o_p( 1 )$.
		Furthermore, we can let 
		\[
		\mathcal{R}_2 = 
		- \frac{1}{2}
		\left[ \bar{\mathcal{E}}^{-1} + \mathcal{R}_1 \right]
		\left[
		\mathcal{F}[ \hat{V}_{Z,K}  ]
		\right]
		\]
		and from Lemma \ref{supp:lem:third:derivative:bounded} and because $\hat{V}_{Z,K} \to 0$, we obtain $\mathcal{R}_2 = O_p( || \hat{V}_{Z,K} ||^2 )$. Hence, we obtain
		\begin{equation} \label{supp:eq:in:proof:TCL:Z:U:taylor}
			\sqrt{K}
			\hat{V}_{Z,K}
			+
			\sqrt{K}
			\mathcal{R}_2
			=
			\left[ \bar{\mathcal{E}}^{-1} + \mathcal{R}_1 \right]
			\left[ \sqrt{K} \nabla_0 \right].
		\end{equation}
		We have $ K^{1/2} \nabla_0 = K^{1/2} \bar{\nabla}_0 + o_p(1) = O_p(1)$ from Lemma \ref{supp:lem:equi:gradient}. Indeed, Lemma \ref{supp:lem:TCL:gradient} implies that $ K^{1/2} \bar{\nabla}_0 = O_p(1)$. Coming back to \eqref{supp:eq:in:proof:TCL:Z:U:taylor}, this implies
		\[
		\sqrt{K}
		\hat{V}_{Z,K}
		=
		\bar{\mathcal{E}}^{-1}
		\left[
		\sqrt{K} \bar{\nabla}_0 
		\right]
		+
		o_p(1).
		\]
		Then from Lemma \ref{supp:lem:exp:matrix}, there exists $\hat{G}_K \in \mathcal{G}_p$ such that
		\[
		\sqrt{K}
		\hat{G}_K
		\hat{U}_{Z,K}
		=
		\sqrt{K}
		\exp 
		\left(
		S
		\left(
		\bar{\mathcal{E}}^{-1}
		\left[
		\bar{\nabla}_0 
		\right]
		+
		o_p(K^{-1/2})
		\right)
		\right).
		\]
		Because the differential of $\exp$ around $0$ is identity, we obtain
		\begin{equation} \label{supp:eq:taylor:final:hat:v}
			\sqrt{K}
			\left(
			\hat{G}_K
			\hat{U}_{Z,K}
			-
			I_p
			\right)
			=
			S
			\left(
			\bar{\mathcal{E}}^{-1}
			\left[
			\sqrt{K} \bar{\nabla}_0 
			\right]
			\right)
			+
			o_p(1).
		\end{equation}
		Thus, the proof is concluded from the central limit theorem on $\sqrt{K} \bar{\nabla}_0 $ obtained in Lemma \ref{supp:lem:TCL:gradient}.
	\end{proof}
	
	\begin{thm} \label{supp:thm:joint:CLT:Z}
		Assume that conditions  \ref{supp:cond:block:independence}, \ref{supp:cond:max:variance} and \ref{supp:cond:asymptotically:distinct:eigenvalues} hold. 
		Let $\Sigma_{\hat{\bar{\mathrm{cov}}}_{Z,K}}$ be the $p^2 \times p^2$ covariance matrix of $ K^{1/2}( \hat{\bar{\mathrm{cov}}}_{Z,K} - I_p)$.  Let $\Sigma_{\text{cross}}$ be the $p^2 \times p^2$ cross covariance matrix between $S
		\left(
		\bar{\mathcal{E}}^{-1}
		\left[
		\sqrt{K} \bar{\nabla}_0 
		\right]
		\right)$ and $ K^{1/2} (\hat{\bar{\mathrm{cov}}}_{Z,K} - I_p)$. 
		Let
		\[
		\Sigma_{ \hat{W}_{Z,K} }
		=
		\Sigma_{ \hat{U} }
		+
		\frac{1}{4}
		\Sigma_{\hat{\bar{\mathrm{cov}}}_{Z,K}}
		-
		\Sigma_{\text{cross}}.
		\] 
		Then for any sequence $\hat{U}_{Z,K}$ in \eqref{supp:eq:ref:hatUZK}, there exists a sequence $\hat{G}_K \in \mathcal{G}_p$ such that, with $Q_{\hat{W}_{Z,K} }$ the distribution of $ K^{1/2} ( \hat{G}_K \hat{W}_{Z,K} - I_p )$, we have
		\[
		d_w 
		\left( 
		Q_{\hat{W}_{Z,K} } , \mathcal{N}(0,\Sigma_{ \hat{W}_{Z,K} }) 
		\right)
		\underset{K \to \infty}{\overset{}{\to}} 0.
		\]
		Furthermore, the matrix $\Sigma_{ \hat{W}_{Z,K} }$ is bounded as $K \to \infty$.
	\end{thm}
	
	\begin{proof}
		The matrix $\Sigma_{ \hat{U} }$ is bounded from Theorem \ref{supp:thm:TCL}. The matrix $\Sigma_{\hat{\bar{\mathrm{cov}}}_{Z,K}}$ is bounded because $\hat{\bar{\mathrm{cov}}}_{Z,K}$ is an average of random matrices with bounded moments from Lemma \ref{supp:lem:basic:bounds} and such that two of them are independent if their index difference is larger or equal to $L$. Hence, the matrix $\Sigma_{ \hat{W}_{Z,K} }$ is bounded.

		Let $\hat{G}_K \in \mathcal{G}_p$ be such that \eqref{supp:eq:taylor:final:hat:v} holds.
		We consider the event where \eqref{supp:eq:taylor:final:hat:v} holds for the rest of the proof, which probability goes to $1$ as $K \to \infty$.
		We have
		\begin{align*}
			\sqrt{K} ( \hat{G}_K \hat{W}_{Z,K} - I_p )
			& =
			\sqrt{K} ( \hat{G}_K \hat{U}_{Z,K}  \hat{\bar{\mathrm{cov}}}_{Z,K}^{-1/2} - I_p )
			\\
			& = 
			\sqrt{K} \hat{G}_K \hat{U}_{Z,K} (\hat{\bar{\mathrm{cov}}}_{Z,K}^{-1/2} - I_p)
			+  
			\sqrt{K} (
			\hat{G}_K \hat{U}_{Z,K}
			-
			I_p      ).
		\end{align*} 
		From \eqref{supp:eq:taylor:final:hat:v}, Lemma \ref{supp:lem:basic:bounds} and Theorem \ref{supp:thm:consistency}, this yields
		\[
		\sqrt{K} ( \hat{G}_K \hat{W}_{Z,K} - I_p )
		=
		-\frac{1}{2}
		\sqrt{K}  (\hat{\bar{\mathrm{cov}}}_{Z,K} - I_p)
		+
		\sqrt{K} (
		\hat{G}_K \hat{U}_{Z,K}
		-
		I_p      )
		+o_p(1).
		\]
		From \eqref{supp:eq:taylor:final:hat:v}, this yields
		\begin{align} \label{supp:eq:for:joint:CLT}
			\sqrt{K} ( \hat{G}_K \hat{W}_{Z,K} - I_p )
			& =
			-\frac{1}{2}
			\sqrt{K}  (\hat{\bar{\mathrm{cov}}}_{Z,K} - I_p)
			+ 
			S
			\left(
			\bar{\mathcal{E}}^{-1}
			\left[
			\sqrt{K} \bar{\nabla}_0 
			\right]
			\right)
			+o_p(1).
		\end{align}
		Then, the quantity in \eqref{supp:eq:for:joint:CLT} is a linear function with bounded coefficients of the pair
		\[
		\begin{pmatrix}
			\sqrt{K} (\hat{\bar{\mathrm{cov}}}_{Z,K} - I_p  )
			\\
			\sqrt{K} \bar{\nabla}_0 
		\end{pmatrix}.
		\]
		We can show that this vector (in dimension $p^2 + p(p-1)/2$) is asymptotically Gaussian, exactly as in the proof of Lemma \ref{supp:lem:TCL:gradient}. Indeed, with the notation of this lemma, the term $(\hat{\bar{\mathrm{cov}}}_{Z,K} - I_p  )$ will contribute to the quantities $V_i^{(j,k)}$ (which definition can be extended also to the case $j \geq k$). Then the quantity in \eqref{supp:eq:for:joint:CLT} is also asymptotically Gaussian. Furthermore the mean vectors of the first two summands on the right hand side of  \eqref{supp:eq:for:joint:CLT} are zero from the proof of Lemma \ref{supp:lem:TCL:gradient} and since $\hat{\bar{\mathrm{cov}}}_{Z,K}$ has mean $I_p$. The covariance matrix of the sum of these two summands is given by $\Sigma_{\hat{W}_{Z,K}}$, which concludes the proof.
	\end{proof}
	
	\section{Extension of the results from $Z$ to $X$}
	
	We recall that we have the definitions
	\begin{equation} \label{supp:eq:ref:hatUXK}
		\hat{U}_{X,K} 
		\in
		\mathrm{argmax}_{ U \in \mathcal{O}_p }
		\sum_{i=1}^K
		||  
		\mathrm{diag}
		\left(
		U
		\hat{\bar{\mathrm{cov}}}_{X,K}^{-1/2}
		\hat{\mathrm{cov}}_{X,i}
		\hat{\bar{\mathrm{cov}}}_{X,K}^{-1/2}
		U'
		\right)
		||^2.
	\end{equation}
	We recall 
	\[
	\hat{W}_{X,K} =
	\hat{U}_{X,K} 
	\hat{\bar{\mathrm{cov}}}_{X,K}^{-1/2}.
	\]
	
	\begin{cor}\label{supp:cor:consistency:X}
		Assume that conditions  \ref{supp:cond:block:independence}, \ref{supp:cond:max:variance} and \ref{supp:cond:asymptotically:distinct:eigenvalues} hold.
		Then for any sequence $\hat{U}_{X,K}$ in \eqref{supp:eq:ref:hatUXK}, there exists a sequence $\hat{G}_K \in \mathcal{G}_p$ such that
		\[
		\hat{G}_K \hat{W}_{X,K}
		\underset{K \to \infty}{\overset{p}{\to}}
		A^{-1}.
		\] 
	\end{cor}
	
	\begin{proof}
		Let $\hat{U}_{X,K}$ in \eqref{supp:eq:ref:hatUXK}. Then, from Lemma \ref{supp:lem:equivariance}, there exists $\hat{U}_{Z,K}$ in \eqref{supp:eq:ref:hatUZK} such that
		\[
		\hat{W}_{X,K} = \hat{W}_{Z,K} A^{-1}.
		\]
		Furthermore, let $\hat{G}_K \in \mathcal{G}_p$ be such that the conclusion of Theorem \ref{supp:thm:consistency} holds. We have
		\begin{align*}
			\hat{G}_K  \hat{W}_{X,K}
			& =
			\left( \hat{G}_K  \hat{W}_{Z,K} \right) A^{-1}
			\\
			& \overset{p}{\to} I_p A^{-1}
		\end{align*}
		as $K \to \infty$.
	\end{proof}
	
	\begin{cor} \label{supp:cor:CLT:X}
		Assume that conditions  \ref{supp:cond:block:independence}, \ref{supp:cond:max:variance} and \ref{supp:cond:asymptotically:distinct:eigenvalues} hold. Let $\Sigma_{\hat{W}_{X,K}}$ be the $p^2 \times p^2$ covariance matrix of the random matrix 
		\[
		M A^{-1},
		\]
		where $M$ is a $p \times p$ random matrix with $p^2 \times p^2$ covariance matrix $\Sigma_{\hat{W}_{Z,K}}$, with the notation of Theorem \ref{supp:thm:joint:CLT:Z}. 
		Then for any sequence $\hat{U}_{X,K}$ in \eqref{supp:eq:ref:hatUXK}, there exists a sequence $\hat{G}_K \in \mathcal{G}_p$ such that, with $Q_{\hat{W}_{X,K} }$ the distribution of $ K^{1/2} ( \hat{G}_K \hat{W}_{X,K} - A^{-1} )$, we have
		\[
		d_w 
		\left( 
		Q_{\hat{W}_{X,K} } , \mathcal{N}(0,\Sigma_{ \hat{W}_{X,K} }) 
		\right)
		\underset{K \to \infty}{\overset{}{\to}} 0.
		\]
		Furthermore, the matrix $\Sigma_{ \hat{W}_{X,K} }$ is bounded as $K \to \infty$.
	\end{cor}
	\begin{proof}
		The matrix $\Sigma_{ \hat{W}_{X,K} }$ is bounded because $A^{-1}$ is fixed and the matrix $\Sigma_{ \hat{W}_{Z,K} }$ is bounded.
		Let $\hat{U}_{X,K}$ in \eqref{supp:eq:ref:hatUXK}. Then, from Lemma \ref{supp:lem:equivariance}, there exists $\hat{U}_{Z,K}$ in \eqref{supp:eq:ref:hatUZK} such that
		\[
		\hat{W}_{X,K} = \hat{W}_{Z,K} A^{-1}.
		\]
		Furthermore, let $\hat{G}_K \in \mathcal{G}_p$ be such that the conclusion of Theorem \ref{supp:thm:joint:CLT:Z} holds. We have
		\begin{align*}
			\sqrt{K} \left( \hat{G}_K  \hat{W}_{X,K}
			- A^{-1} \right)
			& =
			\sqrt{K} \left( \hat{G}_K  \hat{W}_{Z,K}  A^{-1}
			- A^{-1} \right)
			\\
			& = 
			\sqrt{K} \left( \hat{G}_K  \hat{W}_{Z,K} - I_p \right) A^{-1}.
		\end{align*}
		This last quantity follows the asymptotic Gaussian distribution given in the statement of the corollary, which concludes the proof.
	\end{proof}

	\section{Extension to non-zero mean and to empirical centering}
	\label{supp:section:non-zero:mean}
	
	In this section, we consider that $Z$ has a non-zero mean function, that is constant within the time blocks (Condition~\ref{cond:mean:function}).
	\begin{cond} \label{supp:cond:mean:function}
		For any $i \in \{1,\ldots,K\}$ and $j \in \{1,\ldots,s\}$,  the mean vector of $Z_{(i-1)s+j}$ depends only on $i$ and is written $m_i$.
	\end{cond}
	
	In this section, we let, for $i = 1 ,\ldots , K$,
	\[
	\bar{Z}_i
	=
	\frac{1}{s}
	\sum_{j=1}^s
	Z_{(i-1)s+j}.
	\] 
	We also let $\bar{Z}_i = (\bar{Z}_i^{(1)} , \ldots , \bar{Z}_i^{(p)})'$.
	We let for $i \in \{1,\ldots,K\}$,
	\begin{equation} \label{supp:eq:hat:cov:Z:i:centered}
		\hat{\mathrm{cov}}_{Z,i}
		= \frac{1}{s} \sum_{j=1}^s ( Z_{s(i-1)+j} - \bar{Z}_i ) (  Z_{s(i-1)+j} - \bar{Z}_i )' 
	\end{equation}
	and 
	\[
	\mathrm{cov}_{Z,i}
	= \mathbb{E} \left(  \hat{\mathrm{cov}}_{Z,i} \right)
	=
	\frac{1}{s} \sum_{j=1}^s
	\mathrm{Cov}( Z_{s(i-1)+j} - \bar{Z}_i  ),
	\]
	from Condition \ref{supp:cond:mean:function}.
	We also let
	\[
	\bar{\mathrm{cov}}_{Z,K} = \frac{1}{K} \sum_{i=1}^K \mathrm{cov}_{Z,i}
	\]
	and
	\[
	\hat{\bar{\mathrm{cov}}}_{Z,K} = \frac{1}{K} \sum_{i=1}^K \hat{\mathrm{cov}}_{Z,i}.
	\]
	We recall that we assume $\bar{\mathrm{cov}}_{Z,K}  = I_p$, which can always be done by multiplying each component $Z^{(k)}$ by a constant. This is necessary to obtain the identifiability of $A^{-1}$ up to permutations and sign changes of the rows.
	
	We define $\hat{\mathrm{cov}}_{X,i}$, $\mathrm{cov}_{X,i}$, $\bar{\mathrm{cov}}_{X,K}$ and $\hat{\bar{\mathrm{cov}}}_{X,K}$ similarly as $\hat{\mathrm{cov}}_{Z,i}$, $\mathrm{cov}_{Z,i}$, $\bar{\mathrm{cov}}_{Z,K}$ and $\hat{\bar{\mathrm{cov}}}_{Z,K}$ but with $Z$ replaced by $X$.
	Then, $\hat{U}_{Z,K}$, $\hat{W}_{Z,K}$, $\hat{U}_{X,K}$ and $\hat{W}_{X,K}$ are defined as in Section \ref{supp:section:setting:notation}, but with the new definitions of $\hat{\mathrm{cov}}_{X,i}$, $\hat{\bar{\mathrm{cov}}}_{X,K}$, $\hat{\mathrm{cov}}_{Z,i}$ and $\hat{\bar{\mathrm{cov}}}_{Z,K}$ given here.

	We still assume that Condition \ref{supp:cond:block:independence} holds. We assume that the following condition holds, which is a minor change to Condition \ref{supp:cond:max:variance}.
	
	\begin{cond} \label{supp:cond:max:variance:bis}
		We have
		\[
		\sup_{i \in \mathbb{N}}
		\max_{ j=1,...,p }
		\mathrm{Var}
		\left(
		Z_{i}^{(j)}
		\right)
		\leq
		C_{sup}.
		\]
	\end{cond}
	
	We also update some notation from Lemma \ref{supp:lem:calcul:esp:cost}. We let 
	for $i=1,...,K$ and $a,b=1,...,s$, $D_{Z,i}^{(a,b)}$ be the $p \times p$ diagonal matrix defined by
	\[
	\left[
	D_{Z,i}^{(a,b)}
	\right]_{k,k}
	=
	\mathbb{E}
	\left(
	(Z_{(i-1)s+a}^{(k)} - \bar{Z}^{(k)}_i)
	(Z_{(i-1)s+b}^{(k)} - \bar{Z}^{(k)}_i) 
	\right).
	\]
	
	We assume that Condition \ref{supp:cond:asymptotically:distinct:eigenvalues} holds, with the new definition of $\mathrm{cov}_{Z,i}$.
	We let $H$ be defined as in Lemma \ref{supp:lem:non:zero:hessian} but with respect to the new definition of 
	$\mathrm{cov}_{Z,i}$ and $D_{Z,i}^{(a,b)}$.
	
	Under these assumptions, the consistency and the central limit theorem for $\hat{W}_{Z,K}$ and $\hat{W}_{X,K}$ can be extended from the zero-mean and no-centering case, to the case of Condition \ref{supp:cond:mean:function} and of \eqref{supp:eq:hat:cov:Z:i:centered}.
	
	\begin{thm} \label{supp:thm:CLT:non:zero:mean}
		Under the conditions of Section \ref{supp:section:non-zero:mean}, the same conclusions as in Theorems \ref{supp:thm:consistency} and \ref{supp:thm:joint:CLT:Z} and in Corollaries \ref{supp:cor:consistency:X} and \ref{supp:cor:CLT:X} hold, where the definitions of $\Sigma_{\hat{W}_{Z,K}}$ and $\Sigma_{\hat{W}_{X,K}}$ are updated according to the new definitions of $\hat{\mathrm{cov}}_{Z,i}$, $\mathrm{cov}_{Z,i}$ and $D_{Z,i}^{(a,b)}$.
	\end{thm}
	\begin{proof}
		We consider the multivariate time series $Y$ and $W$, defined by, for $i \in \{1,\ldots,K\}$ and $j\in \{1,\ldots,s\}$
		\[
		Y_{(i-1)s+j} = Z_{(i-1)s + j} - \bar{Z}_i
		\]  
		and 
		\[
		W = A Y.
		\]
		We remark that for $i \in \{1,\ldots,K\}$ and $j \in \{ 1, \ldots , s \}$ we have
		\[
		W_{(i-1)s+j}
		=
		X_{(i-1)s+j}
		-
		\bar{X}_i,
		\]
		where we let 
		\[
		\bar{X}_i
		=
		\frac{1}{s}
		\sum_{j=1}^s
		X_{(i-1)s+j}. 
		\]
		These multivariate time series have mean zero from Condition \ref{supp:cond:mean:function} and they are Gaussian. One can check that they satisfy the conditions of Theorems \ref{supp:thm:consistency} and \ref{supp:thm:joint:CLT:Z}. Furthermore, we have
		\[
		\hat{\mathrm{cov}}_{Z,i}
		= \frac{1}{s} \sum_{j=1}^s Y_{s(i-1)+j}  Y_{s(i-1)+j}' 
		\]
		and 
		\[
		\mathrm{cov}_{Z,i}
		=
		\frac{1}{s} \sum_{j=1}^s 
		\mathbb{E}
		\left(
		Y_{s(i-1)+j}  Y_{s(i-1)+j} ' 
		\right).
		\]
		Hence, the conclusion of Theorems \ref{supp:thm:consistency} and \ref{supp:thm:joint:CLT:Z} and of Corollaries \ref{supp:cor:consistency:X} and \ref{supp:cor:CLT:X} applied to $Y$ and $W$ imply Theorem \ref{supp:thm:CLT:non:zero:mean}.
	\end{proof}

\end{document}